\documentclass[12pt]{scrartcl}

  \usepackage{amsmath,amsxtra,amscd,amssymb,latexsym,stmaryrd,amsthm}
  \usepackage{graphicx,color}
  \usepackage[hidelinks]{hyperref}
  \usepackage{appendix}
  \usepackage[utf8]{inputenc}
  \usepackage[T1]{fontenc}
  \usepackage{mathrsfs}
 \usepackage[all]{xy}
  \usepackage{tikz,xcolor}
 \usepackage{authblk}

\addtokomafont{title}{\rmfamily\itshape}

\title{The calculus of thermodynamical formalism}

 \author[$\S$]{Paolo Giulietti\thanks{Supported by Brazilian-French Network in Mathematics}}
 \author[$\ddag$]{Benoit R. Kloeckner\thanks{Supported by the Agence Nationale de la Recherche, grant ANR-11-JS01-0011.}}
 \author[$\S$]{Artur O. Lopes\thanks{Supported by CNPq-Brazil and INCTMat.}}
 \author[$\S$]{Diego Marcon\thanks{Partially supported by CNPq-Brazil through the postdoctoral scholarship PDJ/501839/2013-5.}}
 \affil[$\S$]{Universidade Federal do Rio Grande do Sul, IME, 91509-900, Porto Alegre, Brazil}
 \affil[$\ddag$]{Universit\'e Paris-Est, Laboratoire d'Analyse et de Mat\'ematiques Appliqu\'ees (UMR 8050), UPEM, UPEC, CNRS, F-94010, Cr\'eteil, France}

\theoremstyle{plain}
\newtheorem{theo}{Theorem}[section]
\newtheorem{coro}[theo]{Corollary}
\newtheorem{prop}[theo]{Proposition}
\newtheorem{lemm}[theo]{Lemma}
\newtheorem{theomain}{Theorem}
\newtheorem{coromain}[theomain]{Corollary}
\newtheorem{exemain}[theomain]{Example}

\theoremstyle{definition}
\newtheorem{defi}[theo]{Definition}
\newtheorem{rema}[theo]{Remark}
\newtheorem{exem}[theo]{Example}
\newtheorem{ques}[theo]{Question}

\newcommand{\fspace}[1]{\mathcal{#1}}
\newcommand{\spacem}[1]{\mathcal{#1}}
\newcommand{\op}[1]{\mathscr{#1}}
\newcommand{\dd}{\mathrm{d}}
\newcommand{\supp}{\operatorname{supp}}
\newcommand{\Var}{\operatorname{Var}}
\newcommand{\entropy}{\operatorname{h_{\fspace{X}}}}
\newcommand{\pressure}{\operatorname{Pr}}

\newcommand{\Id}{\operatorname{Id}}
\newcommand{\interior}{\operatorname{int}}
\newcommand{\rv}{\operatorname{rv}}
\newcommand{\Rot}{\operatorname{Rot}}
\newcommand{\Holder}{\operatorname{Hol}}

\newcommand{\fr}{\partial}

\newlength{\hypobox}
\newcounter{hypo}
\renewcommand{\thehypo}{(H\arabic{hypo})}
\newenvironment{hypo}{\refstepcounter{hypo}\medskip\par\noindent%
  \begin{minipage}[t]{1.5em}\textbf{\thehypo}\end{minipage}%
  \hfill\begin{minipage}[t]{\hypobox}}{\end{minipage}\medskip}

\begin{document}
    \makeatletter
    \def\blfootnote{\gdef\@thefnmark{}\@footnotetext}
    \let\@fnsymbol\@roman
    \makeatother

\maketitle

\begin{abstract}
Given a finite-to-one map $T$ acting on a compact metric space $\Omega$
and an appropriate Banach space of functions $\fspace{X}(\Omega)$, one
classically constructs for each potential $A \in \fspace{X}$
a transfer operator $\op{L}_A$ acting on $\fspace{X}(\Omega)$. Under suitable
hypotheses, it is well-known that
$\op{L}_A$ has  a maximal eigenvalue $\lambda_A$, has a spectral gap and defines a unique Gibbs measure $\mu_A$. Moreover there is a unique
normalized potential of the form $B:=A+f-f\circ T+c$ acting as a
representative
of the class of all potentials defining the same Gibbs measure.

The goal of the present article is to study the geometry of the set
of normalized potentials $\fspace{N}$, of the normalization map
$A\mapsto B$, and of the Gibbs map $A\mapsto \mu_A$. We give
an easy proof of the fact that $\fspace{N}$
is an analytic submanifold of $\fspace{X}$ and that the normalization
map is analytic; we compute the derivative
of the Gibbs map; last we endow $\fspace{N}$ with a natural weak
Riemannian metric (derived from the asymptotic variance) with respect to which we compute the gradient
flow induced by the pressure with respect to a given potential,
e.g. the metric entropy functional.
We also apply these ideas to recover in a wide setting existence and
uniqueness of equilibrium states, possibly under constraints.
\end{abstract}

\blfootnote{\textup{2010} \textit{Mathematics Subject Classification}: 37D35, 37A35, 37C30, 49Q20}

\section{Introduction}

The goal of this article is to propose a differential-geometric
approach to the thermodynamical formalism for maps whose
transfer operators satisfy the conclusion of the Ruelle-Perron-Frobenius
theorem (for example, expanding maps).

Some of our results stated below are already known for certain dynamical
systems (see later for more precise historical references); let us stress
what we believe are our main contributions:
\begin{itemize}
\item we propose a point of view, based on differential geometry in
  the space of potentials, which provides new and efficient\footnote{See for
  example Corollary \ref{coro:eigendata}.} proofs
   of strong results
  (e.g. Fr\'echet dérivatives are computed instead of mere directional
  derivatives) valid in a fairly general framework,
\item we compute an explicit formula for the derivative of
  $\int \varphi\,\dd\mu_A$ with respect to $A$ (Theorem \ref{theomain:DG}),
  leading naturally to the variance metric linking together several
  natural quantities (Theorem \ref{theomain:metric}),
\item we use this metric to define the gradient of natural
  functionals, which leads to a gradient flow  modeling
  a system out of equilibrium (Section \ref{sec:gradient-flow}),
\item we show that the map sending a potential to its Gibbs measure
  is very far from being smooth in the sense of optimal transportation
  (Theorem \ref{theomain:roughness}),
\item we improve a result of Kucherenko and Wolf, identifying
  precisely the equilibrium state of a potential under a finite
  set of linear constraints (Theorem \ref{theomain:constraints},
  see also Example \ref{exemain:constraints}).
\end{itemize}

\subsection{Transfer operator, Gibbs measures and normalization}

Let $\Omega$ be a compact metric space and $T:\Omega\to\Omega$
be a finite-to-one map, defining a discrete-time dynamical system.
The model cases
are uniformly expanding maps such as $x\mapsto dx\mod 1$ on
the circle or the shift over right-infinite words on a finite alphabet,
but we shall consider a very general setting by mostly
asking\footnote{see Section \ref{sec:notations} for the precise
hypotheses} that for each potential $A$ in a suitable function
space $\fspace{X}(\Omega)$, the Ruelle-Perron-Frobenius theorem holds for
the transfer operator
\begin{align*}
\op{L}_A &: \fspace{X}(\Omega) \to \fspace{X}(\Omega) \\
    f &: x \mapsto \sum_{T(y) = x} e^{A(y)} f(y)
\end{align*}
i.e. $\op{L}_A$ has a positive, simple
leading eigenvalue
$\lambda_A$ associated with a positive eigenfunction $h_A$; its
dual operator acting on measures $\op{L}_A^\ast$ has a unique
eigenprobability $\nu_A$; and $\op{L}_A$ has a spectral gap below
$\lambda_A$. Then, the measure $\mu_A = h_A\nu_A$ (where the
multiplicative constant in $h_A$ is chosen so as to make $\mu_A$ a
probability measure) is an invariant measure for $T$, which we will call
here the \emph{Gibbs measure} associated to the potential $A$.

Two different potentials $A,B$ which differ by a constant
and a coboundary define the same Gibbs measure, which can thus
be parametrized by the quotient space $\fspace{Q}=\fspace{X}(\Omega)/\fspace{C}$
where
\[ \fspace{C} = \{ c + g - g\circ T \,|\, c\in \mathbb{R},
  g\in\fspace{X}(\Omega) \}. \]
The subset $\fspace{N}\subset \fspace{X}(\Omega)$ of normalized potential
(i.e. such that $\lambda_A=1$ and $h_A = \mathbf{1}$) contains
exactly one representative of each class modulo
$\fspace{C}$, making $\fspace{N}$ another natural parameter space
for Gibbs measures.

Our main object of study is the first-order variation of $\mu_A$ with
respect to $A$, which means we consider
$\mu_{A+\zeta}$ for small $\zeta\in\fspace{X}(\Omega)$;
of course, adding to $\zeta$ a constant and a coboundary will
have no effect. In the literature, it is often the case
that one asks $\zeta$ to satisfy the normalizing condition
$\int \zeta \,\dd\mu_A=0$ to get rid of the constant, and
then considers $\zeta$ up to coboundaries. We argue that instead,
it makes things simpler and clearer to go fully with one of these points
of view: either consider both $A$ and $\zeta$ modulo
$\fspace{C}$, or restrict to normalized $A$ and constrain $\zeta$ to be
tangent to $\fspace{N}$. Our fist result gives a solid ground to this
principle (Theorem \ref{theo:normalized}):
\begin{theomain}\label{theomain:N}
The set $\fspace{N}$ of normalized potentials
is an analytic submanifold of $\fspace{X}(\Omega)$
and its tangent space at $A$ is $\ker \op{L}_A$,
which is a topological complement to $\fspace{C}$.
\end{theomain}

From this we will easily deduce the analyticity
and derivative of several important maps. Consider:
\begin{itemize}
\item the \emph{normalization map} $N:\fspace{X}(\Omega)\to \fspace{N}$
which sends $A$ to the unique normalized potential
in its class modulo $\fspace{C}$,
\item the \emph{leading eigenvalue map} $\Lambda:A\mapsto \lambda_A$,
\item the \emph{leading eigenfunction map} $H:A\mapsto h_A$,
\item the \emph{Gibbs map} $G:A\mapsto \mu_A$ taking its
  values in $\fspace{X}(\Omega)^*$ with the convention that
  $\mu_A(\varphi)=\int\varphi\,\dd\mu_A$.
\end{itemize}

Then we get (Theorem \ref{theo:Nmap} to Proposition \ref{coro:analytic}):
\begin{coromain}\label{coromain:analytic}
The maps $N$, $\Lambda$, $H$, $G$ are analytic and for all
$A\in\fspace{X}(\Omega)$:
\begin{itemize}
\item the differential $DN_A$ of $N$ at $A$ is the linear projection on $\ker \op{L}_{N(A)}$ along $\fspace{C}$,
\item $D(\log\Lambda)_A=\mu_A$ as a linear form on $\fspace{X}(\Omega)$, i.e.
\[\left.\frac{\dd}{\dd t}\log\lambda_{A+t\zeta}\right|_{t=0}
  = \int\zeta \,\dd\mu_A \qquad\forall A,\zeta\in\fspace{X}(\Omega).\]
\end{itemize}
\end{coromain}
The analyticity of these maps and the derivative of $\log\Lambda$
are well-known for many dynamical systems,\footnote{For an historical
account of the problem, see the introduction of \cite{BCV}
and references therein (among others \cite{PP},\cite{Man},\cite{Bar2}).}
however our framework is quite general (e.g. we do not assume
any high-temperature hypothesis) and our method
pretty elementary: we only use basic differential calculus, not
complex analysis nor Kato's theory of regularity of eigendata for
operators  (as done, for example, in \cite{PP}, \cite{SSS}).

\subsection{From integral differentiation to a Riemannian metric}

Both derivatives above are really easy to obtain, but the derivative
of $G$ is slightly more complicated (Theorem \ref{theo:affine-derivative}):
\begin{theomain}\label{theomain:DG}
For all $A,\varphi,\zeta\in\fspace{X}(\Omega)$ we have
\[\frac{\dd}{\dd t} \left. \int \varphi\,\dd \mu_{A+t\zeta} \right|_{t=0}
  = \int (I-\op{L}_{N(A)})^{-1}(\varphi_A) \cdot DN_A(\zeta) \,\dd\mu_A\]
where $\varphi_A=\varphi-\int\varphi\,\dd\mu_A$.
\end{theomain}
(Note that of course the left-hand-side is $DG_A(\zeta)\in\fspace{X}(\Omega)^*$
applied at $\varphi$.)

This derivative can then be expressed in various forms using standard
computations, see Sections \ref{sec:integral} and \ref{sec:metric},
and some interesting connections appear.
\begin{theomain}\label{theomain:metric}
All the following expressions
\begin{align*}
\langle\zeta,\eta\rangle_A &= D^2(\log\Lambda)_A(\zeta,\eta) \\
\langle\zeta,\eta\rangle_A &= \frac{\dd}{\dd t} \left. \int \eta \,\dd\mu_{A+t\zeta}\right|_{t=0} \\
\langle\zeta,\zeta\rangle_A &= \Var(\zeta_A,\mu_A) := \lim_{n \to \infty}
\frac{1}{n} \int \Big(\sum_{i=0}^{n-1} \zeta_A\circ T^i \Big)^2 \,\dd\mu_A\\
\langle\zeta,\eta\rangle_A &= \int \zeta\eta \,\dd\mu_A
  \quad \mbox{whenever } A\in\fspace{N}, \zeta,\eta\in T_A\fspace{N}.
\end{align*}
define the same analytic map $A\mapsto \langle\cdot,\cdot\rangle_A$ from
$\fspace{X}(\Omega)$ to the Banach space of symmetric linear $2$-forms,
such that $\langle\cdot,\cdot\rangle_A$ is positive-semi-definite
with kernel equal to $\fspace{C}$ for all $A$.
This map induces by restriction
a weak Riemannian metric on $\fspace{N}$, and then by projection
it induces a weak Riemannian metric on $\fspace{Q}=\fspace{X}(\Omega)/\fspace{C}$.
\end{theomain}
The metric $\langle\cdot,\cdot\rangle_A$ is thus a close cousin to
McMullen's variance metric introduced in the context of Teichm\"uller
space \cite{Mac}\footnote{See also \cite{BCS} by Bridgeman, Canary and
Sambarino and references therein, and \cite{PS} by Pollicott and Sharp
for an analogous metric of Weyl-Patterson type on spaces of metric graphs.}
(up a conformal rescaling by
entropy), contains the derivative of the Gibbs map, controls
the convexity of $\log\Lambda$ and extends the $L^2(\mu_A)$ metric on
$\fspace{N}$ at the same time.

In the closing Section \ref{sec:2symbols}, we show a concrete example of this approach.  When the dynamics is just the shift on the Bernoulli space
$\{1,2\}^\mathbb{N}$ and the potential depends only on two coordinates, we exhibit the metric explicitly and we compute the curvature (Proposition \ref{metricprop}),
 which is positive. In analogy with the work of McMulleen, one
could conjecture that when our metric is rescaled by the entropy, the curvature is strictly negative, but we show that this is not the case.

\subsection{The optimal transportation approach to the differentiability of
measure-valued maps}

Above we took a very common point of view, considering the Gibbs map
$G:A\mapsto\mu_A$ as taking value in (an affine subspace of) the
Banach space $\fspace{X}(\Omega)^*$, yielding an obvious differential structure
in which each $\varphi\in\fspace{X}(\Omega)$ defines a ``coordinate function''
by $\mu_A\mapsto \mu_A(\varphi) = \int\varphi\,\dd\mu_A$.
We call this the ``affine differential
structure''.

However, this is not the only way to study the regularity of such a map,
and in Section \ref{sec:Was} we study the ``Wasserstein differential structure'' aspect of the question.
One can see $G$ as taking values in the subset $\spacem{P}_T(\Omega)$
of $T$-invariant measures
in the set $\spacem{P}(\Omega)$ of all probability measures, and use
the differential framework based on the $2$-Wasserstein distance $W_2$
from optimal transportation which has been developed in the last fifteen
years.\footnote{The full story does not fit in the bottom margin, but
let us mention the important cornerstones which are the works of
Otto \cite{Otto}, Benamou and Brenier \cite{BB}, and Ambrosio, Gigli
and Savar\'e \cite{AGS}; see also \cite{Vi1}, \cite{Villani2} and \cite{Gigli:structure}.}
This point of view proved useful in the study of the action of expanding circle maps
near the absolutely continuous invariant measure by one of the present authors (see \cite{Kl,Kl2}); here, we show that with the $2$-Wasserstein metric the Gibbs map $A\mapsto \mu_A$ is
far from being differentiable even in the simplest smoothest case.
\begin{theomain}\label{theomain:roughness}
Assume $T=x\mapsto dx \mod 1$ is the standard $d$-self-covering map of the
circle $\mathbb{S}^1=\mathbb{R}/\mathbb{Z}$ and $\fspace{X}(\Omega)$ is the
space of $\alpha$-H\"older functions
for some $\alpha\in(0,1]$. Then
given any smooth path $(A_t)$ in $\fspace{X}(\Omega)$,
the path of Gibbs measures $(\mu_{A_t})$ is not absolutely continuous
in $(\spacem{P}(\Omega),W_2)$ unless it is constant.
\end{theomain}
Recall that a path in a metric space defined on an interval $I$
is said to be \emph{absolutely
continuous} when it has a metric speed in $L^1(I)$: this is
a very weak regularity, so that Theorem \ref{theomain:roughness}
can be interpreted as meaning that a small perturbation of the potential induces a brutal reallocation of the mass distribution
in the sense of $W_2$. This contrasts with
a Lipschitz regularity result obtained for the $1$-Wasserstein metric
in \cite{KLS} (but note that $W_1$ does not yield a differential structure).

\subsection{Applications to equilibrium states}

We end this introduction by presenting some applications and illustrations
of our differential calculus
setting and the metric obtained above.

First we show that it is a nice framework to derive and extend to
our framework the existence and uniqueness of equilibrium states.
Consider the following optimization problems
and induced functionals:
\begin{align*}
\entropy(\mu) &:= \inf_{A\in\fspace{X}(\Omega)} \log\lambda_A-\int A\,\dd\mu
  &\mbox{for }\mu &\in\spacem{P}_T(\Omega) \\
\pressure(B) &:=\sup_{\mu\in\spacem{P}_T(\Omega)}\entropy(\mu)+\int B\,\dd\mu
  &\mbox{for } B &\in\fspace{X}(\Omega)
\end{align*}
(recall that $\fspace{X}(\Omega)$ is a suitable space of potentials
$\Omega\to \mathbb{R}$ which we can chose with some freedom).

\begin{theomain}\label{theomain:pressure}
For all $B\in\fspace{X}(\Omega)$, the supremum in the definition of $\pressure(B)$
is uniquely realized by $\mu_B$ and it holds $\pressure(B)=\log\lambda_B$.
\end{theomain}

We show in Remark \ref{rema:definition2} that for the case of the Classical Thermodynamical Formalism in the sense of \cite{PP}  (the shift acting on the Bernoulli space) and for any invariant probability $\mu$ we have equality between $\entropy (\mu)$ and the metric entropy of $\mu$.  In this case
the pressure $\pressure{}$ defined above also coincides with the usual
topological pressure.
We consider however more general hypothesis in our reasoning. We will refer to $\entropy{}$ and $\pressure{}$ as ``entropy'' and
``pressure'' from now on.

We then observe that the metric $\langle\cdot,\cdot\rangle_A$
enables us to define the gradient of various natural
dynamical quantities, including entropy and pressure
(see Proposition \ref{prop:gradients}). This gives sense to
the gradient flow of the functional
\[A\mapsto \entropy(\mu_A)+\int B \,\dd\mu_A\]
obtained by composing $G$ with the functional defining the pressure.
This gradient flow has a \emph{linear} form when expressed in the quotient
space $\fspace{Q}$ and
can serve as a model for non-equilibrium dynamics, according to which
a system out of
equilibrium behaves just like a system at equilibrium with a
varying potential (Section \ref{sec:gradient-flow}). In case of
a mere change in the temperature of the system's environment, this model
predicts the physically sound property that the systems
evolves only in its temperature (Remark \ref{rema:temperature}).

As a consequence of Theorems \ref{theomain:metric}
and \ref{theomain:pressure}
we obtain several results related to works by Kucherenko and Wolf.
The first result, obtained in \cite{KW1} under somewhat different
hypotheses, is a prescription result.
Given $\Phi=(\varphi_1,\dots,\varphi_K)$ a tuple of test functions
in $\fspace{X}(\Omega)$, the ``rotation vector''
\[\rv(\mu) = \big(\int\varphi_1,\dd\mu,\dots,\int\varphi_K\,\dd\mu)\]
of a $T$-invariant measure describes some convex set
$\Rot(\Phi)\subset \mathbb{R}^K$. The result is then that
for all base potential $B$, every interior value of $\Rot(\Phi)$ can be
realized uniquely as the Gibbs measure of a potential of the
form $B+a_1\varphi_1+\dots+a_K\varphi_K$ (Theorem \ref{theo:prescribing}).

The second result states existence and uniqueness of
equilibrium states under linear constraints;
it is very close to Theorem B of \cite{KW2}, but
even disregarding the difference in our hypotheses we obtain
a more precise description of the equilibrium state:
the parameter $s$ of \cite{KW2} is always equal to $1$. In other words:
\begin{theomain}\label{theomain:constraints}
Let $\Phi=(\varphi_1,\dots,\varphi_K)\in\fspace{X}(\Omega)$ be such that
$0\in \interior \Rot(\Phi)$.
Given any $B\in\fspace{X}(\Omega)$, the restriction
of
\[P_B:\mu\mapsto \entropy(\mu)+\int B\,\dd\mu\]
to the set $\spacem{P}_T[\Phi]$
of  invariant measures realizing $\int\varphi_k\,\dd\mu=0$ for all $k$
is uniquely maximized at the unique Gibbs measure in $\spacem{P}_T[\Phi]$
that is defined by a potential of the form
$B+a_1\varphi_1+\dots+a_K\varphi_K$.
\end{theomain}

We also recover Theorem B in \cite{KW1} (the supremum of entropy of measures
realizing
a given vector in the interior of $\Rot(\Phi)$ depends analytically on the
vector,
Corollary  \ref{coro:KW1}), and by nature our method could be applied to
more general constraints (e.g. asking that $\rv(\mu)$ belongs to
some submanifold of $\Rot(\Phi)$).

Theorem \ref{theomain:constraints} notably
shows that when $T$ is the shift map
and the test functions and the potential $B$ all depend only on $n$
coordinates, so does the potential of the constrained equilibrium state,
which is thus a $(n-1)$-Markovian measure (Remark \ref{rema:constraints},
which also follows from the results of \cite{KW2} but is not stated there).

This result is precise enough to
yield explicit solutions to some concrete maximizing questions,
which as far as we know would be difficult to solve without it.
Let us give a toy example which turns out to have a
nontrivial answer.

\begin{exemain}\label{exemain:constraints}
Assume $T$ is the shift map on $\Omega=\{0,1\}^{\mathbb{N}}$.
Among shift-invariant measures $\mu$ such that
$\mu(01*) = 2\mu(11*)$, the Markov measure associated
to the transition probabilities
\begin{align*}
\mathbb{P}(0\to 0) &= 1-a & \mathbb{P}(0\to1) &= a \\
\mathbb{P}(1\to0) &= \frac23 & \mathbb{P}(1\to1) &= \frac13
\end{align*}
where $a$ is the only real solution to
\[(1-a)^5=\frac{4}{27} a^2 \qquad (a\simeq 0.487803)\]
uniquely maximizes entropy.
\end{exemain}

As a final remark, we mention that optimization
problems such as we solve in  Theorem \ref{theomain:constraints}
appear naturally in multifractal analysis, see
\cite{Bar2}, \cite{Bar1}, \cite{Climenhaga}.
Our approach might lead to explicit computation in that
field.

\section{Notation and preliminaries}\label{sec:notations}

We shall consider the thermodynamical formalism associated with
a discrete-time, con\-ti\-nuous-space dynamical system.
The phase space shall be
denoted by $\Omega$, and will be assumed to be a \emph{compact}
metric space, whose metric shall be denoted by $d$. The time
evolution is then described by a map $T:\Omega\to \Omega$
which will be assumed to be a finite-to-one map.
We will denote by $\spacem{P}(\Omega)$ the set of probability measures
on $\Omega$, and by $\spacem{P}_T(\Omega)$ the subset of
$T$-invariant measures.

Typical cases to serve as examples are the shift on
$\mathcal{A}^{\mathbb{N}}$ where $\mathcal{A}$ is a finite alphabet,
the maps $x\mapsto dx\mod 1$ acting on the circle
$\mathbb{S}^1=\mathbb{R}/\mathbb{Z}$. The reader not willing
to deal with the detailed hypotheses below can
assume $T$ is one of these classical maps.

\begin{rema}
We also want to consider cases such as the tent map
\[x\mapsto \begin{cases} 2x &\mbox{ if }x\le \frac{1}{2}\\
  2-2x &\mbox{ if }x\ge \frac{1}{2} \end{cases}\]
on the interval $[0,1]$.
This map has a particularity shared with other
map of the same kind: one point, $1/2$, has only
one inverse image while its neighboring points have two.
This will make it necessary to adjust some of the definitions
below. Let us formalize a property of the tent map which
we will refer to when we explain these modifications:
the tent map \emph{has local inverse branches} in the sense that
for all $x\in\Omega$ there is an integer $d\ge 2$
(to be implicitly taken minimal),
a neighborhood $V$ of $x$ and continuous maps
$y_k : V_k\subset \Omega \to V$ (where $k\in\{1,\dots, d\}$) such
that for all $x'\in V$ we have
\[T^{-1}(x') = \{y_1(x'),\dots,y_d(x')\}.\]
\end{rema}

\subsection{Working Hypotheses}

\subsubsection{Space of potentials}

The first set of assumptions we make concerns the regularity
of potentials; in designing the hypotheses below we tried to
keep them general enough not to rule out non continuous
potentials; e.g. in some settings bounded variation functions
are meaningful (in particular when $T$ is only piecewise continuous).

We fix for all the article a space of functions
$\fspace{X}(\Omega)$, endowed with a norm $\lVert\cdot\rVert$,
satisfying the following.
\begin{hypo}
$\fspace{X}(\Omega)$ is a Banach space of
Borel-measurable, bounded functions $\Omega\to\mathbb{R}$, which includes
all constant functions; for all $f,g \in\fspace{X}(\Omega)$
we have
\[\lVert fg\rVert \le \lVert f\rVert \ \lVert g \rVert,\]
for all $f\in\fspace{X}(\Omega)$ that is positive and bounded away
from $0$, the function $\log f$ also lies in $\fspace{X}(\Omega)$;
and for some constant $C$ it holds
\[\lVert f \rVert \ge C\sup_{x\in\Omega} |f(x)|.\]
\label{hypo:Banach}
\end{hypo}

In particular for each probability measure
$\mu$ on $\Omega$,
the linear form defined by $f\mapsto \int f\,\dd \mu$ is
continuous: in other words, every probability measure
can be seen as an element of $\fspace{X}(\Omega)^*$.

Note that when $f\in\fspace{X}(\Omega)$ is positive
and bounded away from $0$, $1/f=e^{-\log f}$ is
also in $\fspace{X}(\Omega)$.

\begin{rema}
In some circumstances, one works with a norm
satisfying only the weak multiplicativity condition
$\lVert fg\rVert \le C \lVert f\rVert \ \lVert g \rVert$
for some positive constant $C$. Then one can define
a new, equivalent norm $\lVert\cdot\rVert'=C\lVert\cdot\rVert$
which is then multiplicative.
\end{rema}

\begin{exem}
The space
$\Holder_\alpha(\Omega)$ of $\alpha$-H\"older functions
(for some $\alpha\in(0,1]$) with its usual norm
\[ \lVert f \rVert_{\mathrm{\alpha}} = \sup_{x\in \Omega} |f(x)|
  + \sup_{x\neq y\in\Omega} \frac{|f(x)-f(y)|}{d(x,y)^\alpha}\]
satisfies \ref{hypo:Banach}.
When $\alpha=1$, we get the space $\mathrm{Lip}(\Omega)$
of Lipschitz-continuous functions. Note that $d^\alpha$ is a distance
on $\Omega$, and that $\Holder_\alpha(\Omega)$ coincide with
$\mathrm{Lip}(\Omega,d^\alpha)$.
\end{exem}


Next, we need a compatibility hypothesis between $T$ and
$\fspace{X}(\Omega)$.

\begin{hypo}
$T$ \emph{preserves $\fspace{X}(\Omega)$ forward and backward},
i.e. the composition operator $f\mapsto f\circ T$ is well-defined
and continuous from $\fspace{X}(\Omega)$ to itself, and
or all $f\in\fspace{X}(\Omega)$, we have
   \[g:x\mapsto\sum_{T(y)=x} f(y) \in\fspace{X}(\Omega) \quad
   \mbox{and}\quad \lVert g\rVert \le C\lVert f\rVert \]
for some constant $C$ (i.e. $f\mapsto g$ is a continuous operator
on $\fspace{X}(\Omega)$).
\label{hypo:preserves}
\end{hypo}

\begin{exem}
When $\fspace{X}(\Omega)=\Holder_\alpha(\Omega)$,
it is sufficient to ask the map $T$ to be a local
bi-Lipschitz homeomorphism to obtain \ref{hypo:preserves}.
\end{exem}

\begin{rema}
The tent map does not strictly
speaking satisfy this compatibility when
for example $\fspace{X}(\Omega)=\Holder_\alpha(\Omega)$,
because $\frac{1}{2}$ only
has one inverse image and $g$ is usually not even continuous.
One can fix such cases by
introducing a suitable weight in all sums $\sum_{T(y)=x} f(y)$,
i.e. $\sum_{T(y)=\frac{1}{2}} f(y)$ should be interpreted as  $f(1)+f(1)$
to ensure continuity in $x$ of $\sum_{T(y)=x} f(y)$. In other
words, if needed $\sum_{T(y)=x} f(y)$ can be replaced everywhere by
$\sum_k f(y_k(x))$ where $y_k$ are the local inverse branches of $T$.
\end{rema}

\subsubsection{Transfer operator}

 The composition operator arising from $T$ is the natural functional counterpart to our dynamical
system; in fact, most properties of ergodic flavor of $T$ are naturally
formulated in terms of the composition operator on a certain class
of functions. However it is useful to investigate its ``inverses'',
the transfer operators.
Given a ``potential'' $A\in\fspace{X}(\Omega)$, one defines
a transfer operator (also called a Ruelle operator) by
\[\op{L}_A(f)(x) = \sum_{T(y)=x} e^{A(y)} f(y);\]
note that since $\fspace{X}(\Omega)$ is a Banach algebra,
$e^A$ lies in $\fspace{X}(\Omega)$ and so does $\op{L}_A(f)$;
hypothesis \ref{hypo:preserves} also implies that $\op{L}_A$ is a
continuous operator.

Since $\fspace{X}(\Omega)$ is a space of functions,
it contains a canonical ``positive cone'', the set of positive
functions, which is convex and invariant by dilation.
By design, the transfer operator is positive
in the sense that it maps the positive cone into itself.
Typical expanding assumptions for $T$ ensure that the positive cone
is even mapped into a narrower cone, inducing a contraction
on the set of positive directions endowed with a suitable
distance (see e.g. \cite{Bal}).
Instead of assuming such kind of hypothesis on
$T$, we shall only assume the consequences that are usually drawn from
them. Namely, we ask that $(T,\fspace{X}(\Omega))$
\emph{satisfies a Ruelle-Perron-Frobenius theorem} (including a
spectral gap) in the sense of the following two hypotheses.

\begin{hypo}
For all $A\in \fspace{X}(\Omega)$ the transfer operator
$\op{L}_A$ has a positive maximal eigenvalue $\lambda_A$ and a positive,
bounded away from $0$ eigenfunction $h_A\in\fspace{X}(\Omega)$:
\[\op{L}_A(h_A) = \lambda_A h_A,\]
and the dual operator $\op{L}_A^\ast$ of $\op{L}_A$
preserves the set of finite measures and has a eigenmeasure
$\nu_A\in \spacem{P}(\Omega)$ for the eigenvalue $\lambda_A$, in particular
  \[\int \op{L}_A(f) \,\dd\nu_A = \lambda_A \int f \,\dd\nu_A \qquad
   \forall f\in\fspace{X}(\Omega)\]
\label{hypo:RPFa}
\end{hypo}
Observe that when all functions of $\fspace{X}(\Omega)$
are continuous, $\op{L}_A$ extends to all continuous functions and then
the dual operator automatically acts on measures.
\begin{hypo}
For all $A\in\fspace{X}(\Omega)$, there are positive constants
  $D$, $\delta$ such that
  for all $n\in\mathbb{N}$ and all $f\in\fspace{X}(\Omega)$
   such that $\int f \,\dd\nu_A = 0$, we have
  \[ \lVert \op{L}_A^n(f) \rVert \le
    D \lambda_A^n(1-\delta)^n \lVert f\rVert.\]
\label{hypo:RPFc}
\end{hypo}
It follows in particular that $\lambda_A$ is a simple eigenvalue and
that $\nu_A$ defines a natural (topological) complement to its
eigendirection.

It is easy to see that $\mu_A = h_A \,\nu_A$ defines an invariant measure for $T$, and up to normalizing $h_A$
we can assume $\mu_A$ is a probability measure which we will
call the \emph{Gibbs measure} of $A$.

\begin{exem}
When $T$ is expanding in a relatively general sense
and $\fspace{X}(\Omega)$ is a space of H\"older functions,
\ref{hypo:RPFa} and \ref{hypo:RPFc} are proved in
\cite{KLS} (the spectral gap is
proved there for normalized potentials only, but see remark \ref{rema:gap}).
\end{exem}

\subsubsection{Further hypotheses}

Our first results will only use \ref{hypo:Banach} to \ref{hypo:RPFc},
but at some point we will need two further hypotheses, which feel
harmless (in the sense that they hold for most if not all relevant
examples), but which do not follow from the previous ones.

From Section \ref{sec:metric} on, we will assume:
\begin{hypo}
For all $A,f\in\fspace{X}(\Omega)$, if $f$ is non-negative
and $\int f \,\dd\mu_A=0$ then $f=0$.
\label{hypo:support}
\end{hypo}

\begin{rema}
If all functions in $\fspace{X}(\Omega)$ are continuous,
it is sufficient to ask that $\mu_A$ has full support for all $A$
to ensure \ref{hypo:support}.
\end{rema}

\begin{exem}
Assume that $T$ is continuous, that all functions in $\fspace{X}(\Omega)$
are continuous, and that the only closed subsets $A\subset\Omega$
which are both forward and backward invariant (i.e. $T(A)=T^{-1}(A)= A$)
are the empty set $\varnothing$ and the full space $\Omega$. Then
\ref{hypo:support} holds.

Indeed, since $\mu_A$ is an invariant measure, its support is a closed
invariant subset of $\Omega$. But (assuming without lost of
generality that $A$ is normalized, see below)
the invariance under $\op{L}_A^\ast$
and the fact that $e^A$ is a positive function also implies that
$\supp\mu_A$ is backward invariant, so that $\mu_A$ must have
full support. The continuity of $f$ then gives the conclusion.
\end{exem}

In Section \ref{sec:applications} we will
use the following largeness hypothesis, meant to avoid
degenerate cases such as $\fspace{X}(\Omega)=\{\mbox{constants}\}$.
\begin{hypo}
All continuous functions $f:\Omega\to \mathbb{R}$ can
be uniformly approximated by elements of $\fspace{X}(\Omega)$.
\label{hypo:large}
\end{hypo}
(Note that we do not imply here that the functions in $\fspace{X}(\Omega)$
are continuous themselves.)

\subsection{Normalization} \label{subsec:normalization}

Among the potentials, of particular importance are the \emph{normalized}
ones, i.e. those potentials $A$ such that
$\op{L}_A(\mathbf{1}) = \mathbf{1}$
(where $\mathbf{1}$ denotes the constant function with value $1$)
i.e. such that $\lambda_A=1$ and $h_A=\mathbf{1}$.
In other words, $A$ is normalized when
\begin{equation}
\sum_{T(y)=x} e^A(y) = 1 \quad \forall x\in\Omega.
\label{eq:normalized}
\end{equation}
Two nice properties that give a first evidence for the relevance
this definition  are
that when $A$ is normalized, first $\op{L}_A$ is a left-inverse
to the composition operator:
\[\op{L}_A(f\circ T) = f \quad \forall f\in\fspace{X}(\Omega),\]
second $\op{L}_A^\ast$  preserves the set of probability measures.
One can then interpret $\op{L}_A^*$ as a Markov chain, the
numbers $e^{A(y)}$ representing the probability of transiting
from $x$ to $y$ whenever $T(y)=x$; a realization of this Markov chain
is a random reverse orbit of $T$.

As is well-known, the Ruelle-Perron-Frobenius theorem enables
one to ``normalize'' a potential $A$, by writing
\[B=A+\log h_A - \log h_A\circ T - \log\lambda_A \in\fspace{X}(\Omega).\]
Then one gets
\begin{align*}
\op{L}_B(f)& : x  \mapsto \sum_{T(y)=x} e^{A(y)}
               \frac{h_A(y)}{\lambda_A h_A(x)} f(y) \\
  \op{L}_B(f) & =  \frac{1}{\lambda_A h_A} \op{L}_A(h_A f)
\end{align*}
where \ref{hypo:RPFa}
ensures that $h_A$ is bounded away from $0$ and
\ref{hypo:Banach} then ensures that $1/h_A\in \fspace{X}(\Omega)$.
The transfer operators $\op{L}_A$ and $\op{L}_B$ are thus
conjugated one to another
up to a multiplicative constant $\lambda_A$, the conjugating operator
being the multiplication by $h_A$; in particular
\[\op{L}_B(\mathbf{1}) =\frac{1}{\lambda_Ah_A} \op{L}_A(h_A)
  = \mathbf{1}.\]

This conjugacy shows that
the Gibbs measure $\mu_A = h_A \,\nu_A$
is also the eigenprobability $\nu_B=\mu_B$ of $\op{L}_B^*$; each potential yields an invariant probability measure,
but several potentials can yield the same Gibbs measure.

Using the same computation than above, one sees that
whenever two arbitrary potentials $A$, $B$ are related
as above, i.e. $B= A + g -g\circ T + c$ for some
$g\in\fspace{X}(\Omega)$ and $c\in\mathbb{R}$, then
their transfer operators and their duals are conjugated one to another
up to a constant:
\[\op{L}_B(\cdot)=e^c e^{-g} \op{L}_A(e^g \cdot)
\qquad \op{L}_B^\ast(\cdot) = e^c e^g \op{L}_A^\ast(e^{-g}\cdot)\]
where $e^g$ is in $\fspace{X}(\Omega)$, positive and bounded
away from $0$.
It follows immediately that (up to normalizing constants)
$h_B=h_Ae^{-g}$, $\nu_B=e^g \nu_A$ and $\lambda_B=e^c \lambda_A$.
In particular we have $\mu_B=\mu_A$: both potentials define the same
Gibbs measure. It is also straightforward to check that
if moreover both $A$ and $B$ are normalized, $g$ must be a constant
and $c$ must be zero, so that $A=B$. In other words, we have a subspace
\[\fspace{C} = \big\{g-g\circ T + c \mid g\in\fspace{X}(\Omega), c \mbox{ a constant} \big\} \subset \fspace{X}(\Omega)\]
 such that each class modulo $\fspace{C}$
defines one Gibbs measure, and contains exactly one normalized potential.
One says that a function of the form $g-g\circ T$ is a \emph{coboundary},
thus $\fspace{C}$ is the space generated by coboundaries and constants.
All the above is very classical; our goal is now to study in
more details the following objects:
\begin{itemize}
\item the set $\fspace{N}\subset \fspace{X}(\Omega)$ of normalized potentials (which is not a linear subspace, see Remark \ref{rema:straight}),
\item the \emph{normalization map}
 \[N:A\mapsto A+\log h_A-\log h_A\circ T-\log\lambda_A\]
from $\fspace{X}(\Omega)$ to $\fspace{N}$,
\item the quotient $\fspace{Q}=\fspace{X}(\Omega)/\fspace{C}$, and
\item the Gibbs map $G:A\mapsto \mu_A$ from $\fspace{X}(\Omega)$, seen
as taking value either in $\fspace{X}(\Omega)^*$ or in
$\spacem{P}_T(\Omega)\subset \spacem{P}(\Omega)$
\end{itemize}

The typical questions we want
to answer are of differential-geometric flavor: is $\fspace{N}$
a submanifold of $\fspace{X}(\Omega)$? Are the maps
$N$ and $G$ differentiable? How to endow $\fspace{N}$
or $\fspace{Q}$ with a meaningful Riemannian
metric? Can we then study gradient flows of natural functionals on
these spaces?

\begin{rema}\label{rema:gap}
The conjugacy between the transfer operator of a potential
$A$ and the transfer operator of its normalization $B=N(A)$
shows that a spectral gap for $\op{L}_B$ implies
the same spectral gap for $\op{L}_A$ (with a different constant
$D$, but the same $\delta$). Indeed, if $\int f\,\dd\nu_A=0$
then $\int f/h_A \,\dd\mu_A =0$ and
\begin{align*}
\lVert \op{L}_A^n(f) \rVert
  &= \lVert \lambda_A^n h_A \op{L}_B^n\big(f/h_A\big) \rVert \\
  &\le \lambda_A^n \lVert h_A\rVert D (1-\delta)^n \lVert f/ h_A\rVert \\
  &\le \big(D \lVert h_A\rVert \lVert 1/h_A\rVert\big)\, \lambda_A^n (1-\delta)^n \lVert f\rVert.
\end{align*}
In particular, if hypothesis \ref{hypo:RPFa}
is satisfied, the spectral gap for normalized potentials
implies \ref{hypo:RPFc} for all potentials.
\end{rema}

\begin{rema}
The spectral gap hypothesis implies
an exponential decay of correlation for functions
in $\fspace{X}(\Omega)$: indeed if $A$ is any potential and
$f,g\in\fspace{X}(\Omega)$ are such that $\int f\dd\mu_A=0$, we have
\begin{align*}
\Big| \int f \cdot g\circ T^n \,\dd\mu_A \Big|
  &= \Big| \int \op{L}_{N(A)}^n(f\cdot g\circ T^n) \,\dd\mu_A \Big| \\
  &= \Big| \int \op{L}_{N(A)}^n(f) \cdot g \,\dd\mu_A \Big| \\
  &\le \lVert \op{L}_{N(A)}^n(f) \cdot g \rVert_\infty \\
  &\le C^{-2} \lVert \op{L}_{N(A)}^n(f) \rVert \lVert g\rVert \\
  &\le C^{-2} D (1-\delta)^n \lVert f\rVert \lVert g\rVert
\end{align*}
\end{rema}

\begin{rema} \label{rema:gibbstandard}
In typical situations, a normalized potential $A$ can be recovered
from the Gibbs measure as a Jacobian: for example, if $T$ has inverse
branches $y_i$ near each $x\in\Omega$ which are local homeomorphisms
with disjoint images, then
\[e^{A(y_i(x))} =
\frac{\dd\mu_A(y_i(x))}{\dd\mu_A(x)}\]
in the sense that if $B=B(x,\varepsilon)$ is a small ball around
$x$, the ratio of $\mu_A(y_i(B))$ with respect to $\mu_A(B)$ goes to $e^{A(y_i(x))}$
when $\varepsilon$ goes to zero (in doubt, integrate caracteristic functions of balls with respect to one measure and change variables to verify such claim). Slight adaptations of this
argument are needed for example for tent maps.

In general, it might \emph{a priori} happen that two different normalized
potentials $A,A'$ have the same Gibbs measure. We will see much later in
Remark \ref{rema:gibbsuniqueness} that our assumptions are sufficient
to prevent this, and
ensure perfect identifications between normalized potentials,
mod $\fspace{C}$ classes of potentials, and Gibbs measures.
\end{rema}

\subsection{Analytic maps and submanifolds} \label{sec:imp}

When working in infinite-dimensional spaces,
just as differentiability has various definitions of varying strength
(G\^ateaux versus Fréchet), the analyticity of a map
can be defined in several ways. Here, we take the
strongest definition, recalled below.

First, recall that a closed linear subspace $M$ in a Banach
space $\fspace{X}$ is said to be topologically complemented, or
for short \emph{complemented}, when there is a closed linear subspace $N$
which is an algebraic complement. We shall only write
$\fspace{X}=M\oplus N$ when $M$ and $N$ are topological complements.
The projection to $M$ along $N$ and
the projection to $N$ along $M$ are then continuous, i.e.
for all $x\in\fspace{X}$,
the decomposition $x=m+n$ with $m\in M$ and $n\in N$ exists,
is unique, and $m$ and $n$ depend linearly continuously on $x$.
Equivalently, $M$ is complemented when it is the image, or the
kernel, of a continuous linear projection $\fspace{X}\to\fspace{X}$.

Let $\fspace{X}$ and $\fspace{Y}$ be two Banach spaces,
whose norms will both be denoted by $\lVert\cdot\rVert$.
A continuous, symmetric, multilinear operator
$a: \fspace{X}^k \to \fspace{Y}$ has an operator
norm denoted by $|a|$; if $\zeta$ is a vector
in $\fspace{X}$, we denote by $\zeta^{(k)}$
the element $(\zeta,\zeta,\dots,\zeta)$ of $\fspace{X}^k$
and we have
\[\lVert a(\zeta^{(k)})\rVert\le |a| \lVert\zeta\rVert^k.\]
We shall say that a sequence
$a_k:\fspace{X}^k \to \fspace{Y}$ of such $k$-ary operators
($k\ge 0$) is a series with positive radius of convergence
if the complex series
\[\sum_{k\ge 0} |a_k|z^k \]
has a positive radius of convergence in $\mathbb{C}$.

Let $\Phi:U\subset\fspace{X}\to\fspace{Y}$ be a map defined
from an open subset of $\fspace{X}$. We say that $\Phi$
is \emph{analytic} if for each $x\in U$ there is a series
of $k$-linear, symmetric, continuous operators
$a_k:\fspace{X}^k \to \fspace{Y}$ with positive
radius of convergence such that on an open subset of $U$
the following identity holds:
\[\Phi(x+\zeta) = \sum_{k\ge 0} a_k(\zeta^{(k)}).\]
An analytic map is smooth (in particular, Fr\'echet differentiable
and locally Lipschitz-continuous) and the operators $a_k$
are uniquely defined by $\Phi$. Most classical results
hold in this context, in particular the inverse function
theorem and the implicit function theorem  (see \cite{Chae} and \cite{Whi}).

More precisely, a map $\Phi:U\subset\fspace{X}\to\fspace{Y}$
which is analytic and such that $D\Phi_x:\fspace{X}\to\fspace{Y}$
is a topological isomorphism for each $x$, has a local reciprocal
near each point, which is itself analytic (inverse function theorem);
if a map $F:U\subset\fspace{X}\to\fspace{Y}$ is such that
$F(x)=0$ for some $x\in U$, $DF_x$ is onto $\fspace{Y}$
and $\ker DF_x$ is complemented in $\fspace{X}$,
then the level set $F^{-1}(0)$ is an analytic submanifold of
$\fspace{X}$ in a neighborhood of $x$ (implicit function theorem).
In particular, this means that there is an analytic
diffeomorphism defined in a neighborhood of $x$ that maps
$F^{-1}(0)$ to (an open set of)
a complemented, closed, linear subspace of $\fspace{X}$;
it also means that $F^{-1}(0)$ can be locally written
as the graph of an analytic map over a complemented subspace.

\section{Normalizing potentials}
\label{sec:normalizing}

We will now consider the set of normalized potentials
\[\fspace{N} = \{A\in\fspace{X} \mid \op{L}_A(\mathbf{1})=\mathbf{1} \}\]
and the normalization map $N$ that sends any potential
$A$ to its normalization:
\[N(A)=A-\log\lambda_A + \log h_A -\log h_A\circ T.\]
The map $N$ can be described as the (non-linear) projection on $\fspace{N}$
along
\[\fspace{C}= \big\{g-g\circ T + c \,\big|\, g\in\fspace{X}(\Omega), c \mbox{ a constant} \big\}.\]

We start by a simple Lemma which will both prove useful and
serve as an example of the use of convergent series
in our study. We shall denote by $\ker\mu_A\subset \fspace{X}(\Omega)$
the kernel of $\mu_A$ seen as a linear form, i.e.
\[\ker\mu_A:= \Big\{f\in\fspace{X}(\Omega) \,\Big|\,
  \int f \,\dd\mu_A = 0 \Big\}.\]

\begin{lemm}
If $A$ is a normalized potential, the operator
$I-\op{L}_A$ is onto $\ker\mu_A$, and
its corestriction
to $\ker \mu_A$ has a continuous inverse given by
\[(I-\op{L}_A)^{-1} = \sum_{k=0}^\infty \op{L}_A^k
 \quad :  \quad\ker\mu_A\to \fspace{X}(\Omega).\]
\end{lemm}

\begin{proof}
For all $f\in\fspace{X}(\Omega)$, we have that
\[\int (I-\op{L}_A)(f)\,\dd\mu_A=\int f\,\dd\mu_A
  -\int f \,\dd\big(\op{L}_A^\ast\mu_A\big) = 0,\]
because $\mu_A$ is fixed by $\op{L}_A^\ast$.
It follows that $I-\op{L}_A$ takes its values in $\ker\mu_A$.

For all $f\in\ker\mu_A$ and all $n$, we have
\[(I-\op{L}_A)\bigg(\sum_{k=0}^n \op{L}_A^k(f)\bigg)
  = f - \op{L}_A^{n+1}(f).\]
By the spectral gap assumption, we obtain that
$\sum \op{L}_A^k(f)$ converges, that it is bounded by
 $\frac{D}{\delta}\lVert f\rVert$, and that $\op{L}_A^{n+1}(f)$
goes to $0$ when $n\to\infty$. We deduce that
$\sum_{k=0}^\infty \op{L}_A^k$ is well-defined and a right-inverse
to $I-\op{L}_A$, which is therefore onto $\ker\mu_A$.

By commutation the above shows that, for all $f\in\fspace{X}(\Omega)$,
\[\sum_{k=0}^n \op{L}_A^k\big((I-\op{L}_A(f)\big)\]
converges to $f$, so that we have defined an inverse to
(a corestriction of) $I-\op{L}_A$.
\end{proof}

This has useful consequences, which will be better phrased by
introducing another operator related to $A$.
\begin{defi}
Given any \emph{normalized} potential $A$,
let $\op{M}_A$ be the continuous linear operator on $\fspace{X}(\Omega)$
defined by
\[\op{M}_A(f) = -(I-\op{L}_A)^{-1}\circ\op{L}_A(f_A)
  = -\sum_{k=1}^\infty \op{L}_A^k(f_A),\]
where $f_A:=f-\int f\,\dd\mu_A$.
\end{defi}
Observe that $\op{L}_A$ maps $\ker\mu_A$ to itself,
so that $\op{M}_A$ is indeed well-defined
and takes its values in $\ker\mu_A$, and that $\op{M}_A$
commutes with $\op{L}_A$.

\begin{prop}\label{prop:LC}
Let $A$ be a normalized potential. Then:
\begin{enumerate}
\item \label{enumi:coro1} given $f\in\fspace{C}$, there is a
  unique decomposition $f=g-g\circ T+c$
  with $g\in\ker\mu_A$ and $c$ a constant, given by
  \[c=\int f \,\dd\mu_A \quad\mbox{and}\quad
    g=\op{M}_A(f),\]
\item \label{enumi:coro2} the subspace $\fspace{C}$ is closed in $\fspace{X}(\Omega)$, so
that $\fspace{Q}=\fspace{X}(\Omega)/\fspace{C}$ inherits a Banach space
structure from $\lVert \cdot \rVert$,
\item \label{enumi:coro3} $\ker\op{L}_A$ and
$\fspace{C}$ are (topological) complements in $\fspace{X}$,
\item \label{enumi:coro4} $\op{L}_A$ maps $\fspace{C}$ onto
 $\fspace{X}(\Omega)$.
\end{enumerate}
\end{prop}

\begin{proof}
First observe that given any decomposition
$f=g-g\circ T+c$ and any $T$-invariant probability measure $\mu$, we have
\[\int f\,\dd\mu = \int g\,\dd\mu -\int g \,\dd(T_\#\mu) +c = c\]
where $T_\#\mu_A$ is the usual pushforward of the measure $\mu_A$ with
respect to $T$. Since $\mu_A$ is invariant, it follows
that $c$ must equal $\int f\,\dd\mu_A$ and is uniquely defined.

Let us then check that any
$f\in\fspace{C}\cap\ker\op{L}_A$ must vanish. First,
it is easy to see that $\ker\op{L}_A\subset\ker\mu_A$:
\[\int f \,\dd \mu_A = \int f \,\dd\big(\op{L}^*_A\mu_A\big)
  = \int \op{L}_A(f) \,\dd\mu_A = 0.\]
It follows that we can write $f=g-g\circ T$, so that
\[0=\op{L}_A(g-g\circ T) = \op{L}_A(g) - g\]
and $g$ is an eigenfunction of $\op{L}_A$ for the eigenvalue $1$.
Therefore $g$ is constant and $f=0$.

To prove \ref{enumi:coro1}, we write $f=g_1-g_1\circ T+c$ for some
$g_1$ and with $c=\int f\,\dd\mu_A$.
Setting $g=g_1-\int g_1\,\dd\mu_A$, we still have
$f=g-g \circ T +c$ and
\[\op{L}_A(f-c)=\op{L}_A(g-g\circ T)=\op{L}_A(g) - g\]
since $\op{L}_A$ is a left-inverse to the composition operator.
Now, from $g\in\ker\mu_A$ it follows
$g=-(I-\op{L}_A)^{-1}\op{L}_A(f-c)=\op{M}_A(f)$, as claimed.

To prove \ref{enumi:coro2}, consider a
sequence of functions $f_n\in\fspace{C}$ which converges to
$f\in\fspace{X}(\Omega)$. Then using \ref{enumi:coro1},
we can write $f_n=g_n-g_n\circ T + c_n$
where $g_n, c_n$ are images of $f_n$ by continuous operators.
In particular $g_n$ and $c_n$ have limits $g\in\fspace{X}(\Omega)$
and $c\in\mathbb{R}$, so that $f=g -g\circ T+c\in \fspace{C}$.

To prove \ref{enumi:coro3}, since we already know that
$\ker\op{L}_A$ and $\fspace{C}$ intersect trivially,
we consider any  $f\in\fspace{X}(\Omega)$ and let
$c:=\int f \,\dd \mu_A$ and $g=\op{M}_A(f)$.
We have
\[\op{L}_A(g-g\circ T+c) = \op{L}_A(g) - g +c\op{L}_A(\mathbf{1})
  = \op{L}_A(f-c)+c = \op{L}_A(f)\]
where the second equality follows from
$(\op{L}_A-I)\op{M}_A=\op{L}_A$ on $\ker\mu_A$.
It follows that $\ell := f -(g-g\circ T+c)$ is an element
of $\ker \op{L}_A$. The decomposition
\[f = \ell + (g-g\circ T+c)\]
shows that $\fspace{X}(\Omega) = \ker\op{L}_A + \fspace{C}$ and
since both spaces are closed,
$\ker\op{L}_A$ and $\fspace{C}$ are complements.

To prove \ref{enumi:coro4}, let $f\in\fspace{X}(\Omega)$ and set
$c=\int f \,\dd\mu_A$ and $g:= (I-\op{L}_A)^{-1}(c-f)$. Now
$g-g\circ T+c$ is an element of $\fspace{C}$, and we have
\begin{align*}
\op{L}_A(g-g\circ T+c) &= \op{L}_A(g) - g +c \\
  &= -(I-\op{L}_A)(g)+c \\
  &= f-c+c = f.
\end{align*}
\end{proof}

We are know in a position to prove our first main result,
that $\fspace{N}$ is an analytic submanifold of
$\fspace{X}(\Omega)$. This result might be known,
but we did not find a clear statement in the literature,
related statements are often framed into a weaker
definition of analyticity, the identification of the tangent
space seems new, and we obtain the result without resorting
to complex analysis as usually used to prove the regularity
of the eigendata of operators (see appendix V in \cite{PP},
where the weak definition of analyticity should be noted,
and also section 3.3 in \cite{Bar2}).
We shall in fact \emph{deduce} from Theorem \ref{theo:normalized}
that the leading eigenvalue and positive eigenfunction of $\op{L}_A$
both depend analytically on $A$.

\begin{theo}\label{theo:normalized}
The set $\fspace{N}$ of normalized potentials is
an analytic submanifold of $\fspace{X}(\Omega)$, and
its tangent space at $A\in\fspace{N}$ is
$T_A\fspace{N} = \ker \op{L}_A$.
\end{theo}

\begin{proof}
This is a direct consequence of the implicit function
theorem.

Let $F:\fspace{X}(\Omega)\to\fspace{X}(\Omega)$ be the map defined by
\[F(A)(x) = \op{L}_A(\mathbf{1})(x) = \sum_{Ty=x} e^{A(y)}.\]
Then $F$ is analytic, as follows from the analyticity of the
exponential:
\[F(A+\zeta)= \sum_{k\ge 0} D^k F_A(\zeta),\]
where
\[D^k F_A(\zeta_1,\dots,\zeta_k)(x) := \sum_{Ty=x} e^{A(y)}\frac{\prod_{i=1}^k \zeta_i(y)}{k!} \]
defines a series of continuous, symmetric $k$-linear operators
with infinite radius of convergence
(note that we use here the assumptions
that $\fspace{X}(\Omega)$ has multiplicative norm,
and that $\sum_{Ty=x} \zeta(y)$ is in $\fspace{X}(\Omega)$ for all
$\zeta\in\fspace{X}(\Omega)$).

Now, given a potential $A$ and a vector $\zeta$ both in
$\fspace{X}(\Omega)$, we have
\[DF_A(\zeta)=\sum_{Ty=x} e^{A(y)}\zeta(y)=\op{L}_A(\zeta),\]
so that $DF_A=\op{L}_A$; since we know from Proposition
\ref{prop:LC} that $\ker\op{L}_A$ is complemented and
$\op{L}_A$ is onto $\fspace{X}(\Omega)$, we can apply
the implicit function theorem.
\end{proof}

We also get directly the analyticity of the normalization map as explained in the last paragraph of Section \ref{sec:imp}

\begin{theo}\label{theo:Nmap}
The normalization map $N:\fspace{X}(\Omega)\to \fspace{N}$
sending a potential to its normalized version is analytic.
Moreover, its derivative $DN_A$ at a point $A\in\fspace{X}(\Omega)$
is the linear projection on  $T_{N(A)}\fspace{N} = \ker\op{L}_{N(A)}$
in the direction of $\fspace{C}$.
\end{theo}

\begin{proof}
See figure \ref{fig:potentials} for a general picture of the various maps
involved.
Let $\Pi:\fspace{X}(\Omega)\to\fspace{Q}$
be the quotient map; it is a continuous linear map, and in particular
it is analytic. Its restriction $\Pi_{|\fspace{N}}$
to the submanifold $\fspace{N}$ is therefore an analytic map,
and we have for all $A\in\fspace{N}$:
\[ D(\Pi_{|\fspace{N}})_A = \Pi_{|T_A\fspace{N}} = \Pi_{|\ker\op{L}_A}.\]
Since $\ker\op{L}_A$ and $\fspace{C}$ are topological complements,
this differential is invertible with continuous inverse.
The inverse function theorem then ensures that
\[\Pi_{|\fspace{N}}^{-1} : \fspace{Q} \to \fspace{N}\]
is well-defined and analytic. We get the desired result by observing
that
\[N = \Pi_{|\fspace{N}}^{-1} \circ \Pi.\]
\end{proof}

\begin{figure}[tp]\begin{center}
\begin{tikzpicture}[scale=0.9, every node/.style={scale=0.9}]
\draw (0,0) rectangle (6,4) ;
\draw (4.5,4) node [above] {$\fspace{X}(\Omega)$};
\draw [red,thick] plot [smooth] coordinates {(0,1.5) (1,2) (3,1) (6,3)}
  node [right] {$\fspace{N}$} ;
\draw [blue,very thick] (3,4) -- (3,0)
  node [above right] {$\fspace{C}$} ;
\draw (.5,2.01) -- (1.5,2.01) node [right] {$\ker \op{L}_A$} ;
\fill (1,2) circle (.07) node [above] {$A$} ;
\draw [->] (3, -.3) -- (3, -1.7) ;
\draw (3,-1) node [right] {$\Pi$} ;
\draw [red,thick] (0,-2) -- (6,-2) node [right] {$\fspace{Q}$} ;
\fill (1,-2) circle (.07) node [above] {$[A]$} ;
\draw [->] (6.3,2) -- (7.7,2) ;
\draw (7,2) node [above] {$G$} ;
\draw (10,2) ellipse (2 and 1.5) ;
\draw (10,3.5) node [above] {$\spacem{P}(\Omega)$} ;
\draw [red,thick] plot [smooth]
  coordinates {(8.2,1.5) (9,2) (10,1) (11.5,2.5)} ;
\draw [red,thick] (9.7,2) node [above] {Gibbs measures} ;
\fill (9,2) circle (.07) node [below] {$\mu_A$} ;
\draw [->] (5.7,-1.7) -- (8.5,.5) ;
\end{tikzpicture}
\caption{Potentials and Gibbs measures}
\label{fig:potentials}
\end{center}\end{figure}
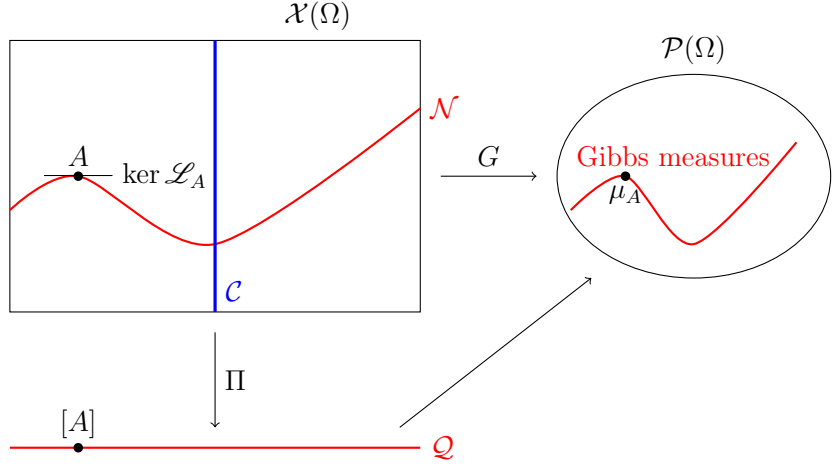

\begin{coro}\label{coro:eigendata}
The maps $\Lambda:\fspace{X}(\Omega)\to\mathbb{R}$
and $H:\fspace{X}(\Omega)\to\fspace{X}(\Omega)$ sending
a potential to its leading eigendata, i.e. defined by
\[\Lambda(A)=\lambda_A \quad\mbox{and}\quad H(A)=h_A\]
(normalized by the condition $\log h_A\in\ker\mu_{A_0}$
for any fixed $A_0$) are analytic maps.

Moreover for all $A,\zeta\in\fspace{X}(\Omega)$:
\[D(\log\Lambda)_A(\zeta) = \int \zeta \,\dd\mu_A.\]
\end{coro}

Note that it will turn out that in our framework $\log\Lambda$
equals the pressure functional, so that this result
gives also the derivative of the later.

\begin{proof}
Fix $A_0$ be any potential, which can be assumed without loss of
generality to be normalized. We then have
\[\Lambda(A)=\exp\big(\int(A-N(A)) \,\dd\mu_{A_0}\big)
\quad\mbox{and}\quad H(A)=\exp\big(\op{M}_{A_0}(A-N(A)\big)\]
which are analytic as composed of analytic maps.

Differentiating $\log\Lambda(A) = \int (A-N(A)) \,\dd\mu_{A_0}$
with respect to $A$ it comes
\[D(\log\Lambda)_A(\zeta) = \int (\zeta - DN_A(\zeta)) \,\dd\mu_{A_0}.\]
This holds for any $A_0$ and any $A$, in particular taking
$A_0=A$ and observing that $DN_A(\zeta)\in \ker\mu_A$
yields the desired formula.
\end{proof}

\begin{coro}\label{coro:analytic}
The map $G:A\mapsto\mu_A\in \fspace{X}(\Omega)^*$ is analytic.
In particular for each $\varphi\in\fspace{X}(\Omega)$, the map
$G_\varphi:A\mapsto \int \varphi \,\dd\mu_A$ is analytic.
\end{coro}

\begin{proof}
Corollary \ref{coro:eigendata} implies that
$\mu_A=D(\log\Lambda)_A$ as a linear form defined on $\fspace{X}(\Omega)$,
so that the Corollary follows from the analyticity
of $\log\Lambda$.
\end{proof}

At this point we have proved Corollary \ref{coromain:analytic}
from the introduction.

Corollaries \ref{coro:eigendata} and \ref{coro:analytic}
where obtained under different assumptions
and with different methods by Bomfim, Castro and Varandas \cite{BCV};
note that we notably do not assume the high-temperature regime
(see their conditions
(P) and (P')) and that once our framework is set, our proofs are
very simple.

\section{Differentiating the Gibbs map in the affine structure}
\label{sec:integral}

There are at least two ways to endow the set of probability measures
$\spacem{P}(\Omega)$ with a kind of differential structure, i.e.
to define what it means for a map such as the Gibbs map $G:A\mapsto \mu_A$
to be differentiable. In this section, we consider the
\emph{affine structure}, while in Section \ref{sec:Was} we will consider
the \emph{Wasserstein structure}.

The affine structure
is obtained simply by observing that $\spacem{P}(\Omega)$ is a convex set
in $\fspace{X}(\Omega)^*$; ``coordinates'' are obtained
by looking at integral of test functions, so that $G$ is often
considered to be differentiable if
\[\forall A,\zeta,\varphi \in\fspace{X}(\Omega): \quad
\frac{\dd}{\dd t} \left. \int \varphi \,\dd\mu_{A+t\zeta} \right|_{t=0}
\mbox{ exists}.\]

We will adopt here the definition of \emph{Fr\'echet} differentiability
for $G:\fspace{X}(\Omega)\to\fspace{X}(\Omega)^*$. It is stronger than the
above one in three respects: we ask
that for each $\varphi$
the directional derivatives at $A$ can be collected as a
\emph{continuous linear} map $\fspace{X}(\Omega)\to \mathbb{R}$,
that all these linear maps for various $\varphi$ can be collected
as a continuous linear map $\fspace{X}(\Omega)\to \fspace{X}(\Omega)^*$,
and that in the Taylor formula defining the derivative,
the remainder is of the form $o(\lVert\varphi\rVert \lVert\zeta\rVert)$
(when $\zeta\to0$).
Note that at this point this strong definition
is already ensured by the analyticity of $G$ and we only want
to get an explicit formula.

\begin{theo}\label{theo:affine-derivative}
For all $A \in\fspace{X}(\Omega)$ there is a neighborhood $U$ of $0$
in $\fspace{X}(\Omega)$ such that for all $\varphi\in\fspace{X}(\Omega)$
and all $\zeta\in U$, we have
\[\int \varphi \,\dd\mu_{A+\zeta} - \int\varphi \,\dd\mu_A =
  \int (I-\op{L}_{N(A)})^{-1}(\varphi_A) \cdot DN_A(\zeta) \,\dd\mu_A
   + O(\lVert\varphi\rVert \lVert \zeta\rVert^2)\]
where $\varphi_A:=\varphi-\int\varphi\,\dd\mu_A$ is the projection
of $\varphi$ on $\ker\mu_A$ along the space of constants.
\end{theo}
Implicitly, the constant in the $O$ depends only on $A$ (and of course
$U,\Omega,T,\fspace{X}(\Omega)$) but not on
$\varphi$ and $\zeta$.
This result will be deduced from the following special
case where the expression is simpler.
\begin{theo}\label{theo:affine-derivative-normalized}
Assume that $A$ is normalized, $\varphi$
has mean $0$ with respect to $\mu_A$, and
$\zeta$ is tangent to $\fspace{N}$ at $A$ and small enough. Then:
\[\int \varphi \,\dd\mu_{A+\zeta} =
  \int (I-\op{L}_A)^{-1}(\varphi) \cdot \zeta \,\dd\mu_A
   + O(\lVert\varphi\rVert \lVert \zeta\rVert^2).\]
\end{theo}
Writing $G_\varphi$ the composition of the evaluation at $\varphi$ and
the Gibbs map, i.e. $G_\varphi(A)=\int\varphi\,\dd\mu_A$,
the above formula can be recast as:
\[D (G_\varphi)_A(\zeta)
 = \int (I-\op{L}_A)^{-1}(\varphi) \cdot \zeta \,\dd\mu_A
  \qquad\mbox{ when }A\in\fspace{N},\, \zeta\in\ker\op{L}_A, \,
  \varphi\in\ker\mu_A.\]

Observe that using the series expression of $(I-\op{L}_A)^{-1}$,
that $\mu_A$ is fixed by $\op{L}_A^\ast$ and that the transfer operator
is a left-inverse to the composition operator, this also rewrites as
\[D (G_\varphi)_A(\zeta)
 = \sum_{i=0}^{+\infty} \int \varphi \cdot \zeta\circ T^i \,\dd\mu_A\]
This version has the advantage that it applies to
test functions $\varphi$ not necessarily in $\ker\mu_A$,
because $\zeta\in\ker\op{L}_A$ implies $\zeta\in\ker\mu_A$
and adding a constant to $\varphi$ does therefore
not change the value of the integrals.

We can rephrase Theorem \ref{theo:affine-derivative}
in a similar way, which will be
used in the sequel to define a metric on $\fspace{X}(\Omega)$.
\begin{coro}\label{coro:affine-derivative}
For all $A,\zeta,\varphi \in\fspace{X}(\Omega)$, if
$\zeta$ is small enough we have
\begin{multline*}
\int \varphi \,\dd\mu_{A+\zeta} - \int\varphi \,\dd\mu_A =\\
  \int \varphi_A \zeta\,\dd\mu_A + \sum_{i=1}^\infty
  \int \big( \varphi_A \cdot \zeta\circ T^i
  +\varphi_A\circ T^i \cdot \zeta\big) \,\dd\mu_A
   + O(\lVert\varphi\rVert \lVert \zeta\rVert^2)
\end{multline*}
where the above sum converges and defines a continuous bilinear form.
\end{coro}

\subsection{The case of a pair of normalized potentials}

To obtain Theorem \ref{theo:affine-derivative-normalized},
thanks to the regularity of the normalization map proved
in the previous section, we are mostly reduced to estimate
$\int\varphi\,\dd(\mu_B-\mu_A)$ when $\varphi\in\fspace{X}(\Omega)$
is fixed and $A,B$ are \emph{normalized} potentials.
Up to adding a constant to $\varphi$, which does not change
the value of the above integral, we assume that $\varphi\in\ker\mu_A$.

We first write (using that $\mu_A$ and $\mu_B$ are respectively fixed
by $\op{L}_A^\ast$ and $\op{L}_B^\ast$)
\begin{align}
\int\varphi\,\dd(\mu_B-\mu_A)
  &= \int \op{L}_B(\varphi) \,\dd\mu_B - \int\op{L}_A(\varphi) \,\dd\mu_A
      \nonumber\\
  &= \int \big(\op{L}_B(\varphi)-\op{L}_A(\varphi)\big) \,\dd\mu_B
     + \int \op{L}_A(\varphi) \,\dd(\mu_B-\mu_A)
     \label{eq:int1}
\end{align}

Then, we observe
\begin{align*}
\big(\op{L}_B(\varphi)-\op{L}_A(\varphi)\big)(x)
  &= \sum_{T(y)=x} e^{A(y)}\varphi(y) \big(e^{B(y)-A(y)}-1\big) \\
  &= \op{L}_A\big(\varphi (e^{B-A}-1)\big)(x),
\end{align*}
so that writing $R(x)=e^x-1-x \sim \frac{1}{2} x^2$, we get
\[\op{L}_B(\varphi)-\op{L}_A(\varphi) = \op{L}_A(\varphi\cdot(B-A))
   + \op{L}_A(\varphi\cdot R(B-A)).\]
Thus:
\begin{align}
\int\varphi\,\dd(\mu_B-\mu_A)
  &= \int \op{L}_A(\varphi\cdot(B-A)) \,\dd\mu_B
     +\int \op{L}_A(\varphi\cdot R(B-A)) \,\dd\mu_B \nonumber\\
  &\qquad   + \int \op{L}_A(\varphi) \,\dd(\mu_B-\mu_A) \nonumber\\
  &= \int \op{L}_A(\varphi\cdot(B-A)) \,\dd\mu_A
    + \int \op{L}_A(\varphi\cdot(B-A)) \,\dd(\mu_B-\mu_A) \nonumber\\
  &\qquad  +\int \op{L}_A(\varphi\cdot R(B-A)) \,\dd\mu_B
       + \int \op{L}_A(\varphi) \,\dd(\mu_B-\mu_A) \nonumber\\
\int\varphi\,\dd(\mu_B-\mu_A)
  &= \int \varphi\cdot(B-A) \,\dd\mu_A
    + \int \op{L}_A(\varphi) \,\dd(\mu_B-\mu_A) \nonumber\\
  & \qquad + \op{I}(\varphi,B) \label{eq:int2}
\end{align}
where
$\op{I}(\varphi,B)=\int \op{L}_A(\varphi\cdot(B-A)) \,\dd(\mu_B-\mu_A)
   +\int \op{L}_A(\varphi\cdot R(B-A)) \,\dd\mu_B$, which
is linear in $\varphi$ and which we now aim at bounding by a multiple
of $\lVert\varphi\rVert \lVert B-A\rVert^2$.

A first tool is the regularity of $G$.
\begin{lemm}
The map $G:\fspace{X}(\Omega)\to\fspace{X}(\Omega)^*$
is locally Lipschitz: for all $A\in\fspace{X}(\Omega)$
there exist a neighborhood $U\in\fspace{X}(\Omega)$ of $A$
and a constant $C$ such that, for all $\varphi\in\fspace{X}(\Omega)$
and all $B\in U$, it holds:
\[\Big| \int\varphi \, \dd(\mu_B-\mu_A) \Big| \le C \lVert \varphi\rVert
  \lVert B-A \rVert.\]
\end{lemm}

\begin{proof}
This follows from the analyticity of $G$ obtained in Corollary
\ref{coro:analytic}.
\end{proof}

A second observation is that since $\fspace{X}(\Omega)$ has a
multiplicative norm, we get
\begin{align*}
\lVert R(B-A) \rVert &= \bigg\lVert \sum_{k\ge 2} \frac{1}{k!}(B-A)^k \bigg\rVert \\
  &\le  \sum_{k\ge 2} \frac{1}{k!} \lVert B-A\rVert^k \\
 & =  R(\lVert B-A\rVert) \\
 & \le C'\lVert B-A\rVert^2
\end{align*}
when $B$ is in any fixed neighborhood $U$ of $A$.

Now, since $\lVert\cdot\rVert$
is assumed to control the sup norm and $\mu_B$ is a probability
measure, whenever $B\in U$ it comes
\begin{align*}
\lvert \op{I}(\varphi,B)\rvert
  &\le C |\op{L}_A| \lVert \varphi\rVert \lVert B-A\rVert^2
   + C''|\op{L}_A| \lVert\varphi\rVert \lVert R(B-A)\rVert \\
  &\le C''' \lVert\varphi\rVert \lVert B-A\rVert^2
\end{align*}

Now, applying \eqref{eq:int2} to its own second term repeatedly
and recalling that $\op{L}_A(\varphi)$ goes to zero
thanks to the spectral gap assumption, we get
\begin{align}
\int\varphi\,\dd(\mu_B-\mu_A)
  &= \int \varphi\cdot(B-A) \,\dd\mu_A
    + \int \op{L}_A(\varphi) \,\dd(\mu_B-\mu_A) \nonumber\\
  & \qquad + \op{I}(\varphi,B) \nonumber \\
  &= \int \big(\varphi+\op{L}_A(\varphi)\Big)\cdot (B-A) \,\dd\mu_A
    + \int \op{L}_A^2(\varphi) \,\dd(\mu_B-\mu_A) \nonumber\\
  & \qquad + \op{I}(\varphi+\op{L}_A(\varphi),B) \nonumber \\
  &= \int \big(\sum_{n\ge0}\op{L}_A^n(\varphi)\Big)\cdot (B-A) \,\dd\mu_A
     + \op{I}\Big(\sum_{n\ge0}\op{L}_A^n(\varphi),B\Big) \nonumber \\
  &= \int (I-\op{L}_A)^{-1}(\varphi) \cdot (B-A)\,\dd\mu_A
     +O(\lVert \varphi\rVert \lVert B-A\rVert^2)\label{eq:int3}
\end{align}
which is almost Theorem \ref{theo:affine-derivative-normalized}, except
for the assumption that $B$ is normalized.

\subsection{End of proofs}

\begin{proof}[Proof of Theorem \ref{theo:affine-derivative-normalized}]
Since $N$ is an analytic projection to $\fspace{N}$ (i.e. $N$ restricted
to $\fspace{N}$ is the identity), we have \[N(A+\zeta)=A+\zeta+O(\lVert\zeta\rVert^2)\]
for all $A\in\fspace{N}$ and all small enough
$\zeta\in T_A\fspace{N}=\ker\op{L}_A$,
with an implicit constant only depending on $A$.

Fix $A\in\fspace{N}$, $\zeta\in \ker\op{L}_A$ and $\varphi\in\ker\mu_A$,
and set $B=N(A+\zeta)$.
Using \eqref{eq:int3} with the normalized potentials $A$ and $B$,
we get
\begin{align*}
\int\varphi\,\dd(\mu_{A+\zeta}-\mu_A)
  &= \int\varphi\,\dd(\mu_{A+\zeta}-\mu_B)
    + \int\varphi\,\dd(\mu_B-\mu_A) \\
  &= O(\lVert\varphi\rVert \lVert A+\zeta-N(A+\zeta)\rVert) \\
  &\qquad  + \int (I-\op{L}_A)^{-1}(\varphi)
    \cdot(B-A)\,\dd\mu_A +O(\lVert\varphi\rVert \Vert B-A \rVert^2)\\
  &=\int (I-\op{L}_A)^{-1}(\varphi)\cdot \zeta\,\dd\mu_A
    +O(\lVert \varphi\rVert \lVert\zeta\rVert^2),
\end{align*}
for $\zeta$ small enough, and with an implicit constant that depends
only on $A$.
\end{proof}

\begin{proof}[Proof of Theorem \ref{theo:affine-derivative}]
Let $A,\zeta,\varphi\in \fspace{X}(\Omega)$ be arbitrary. Then
we consider:
\begin{itemize}
\item  $N(A)$, which is the normalized potential such that
$\mu_{N(A)}=\mu_A$,
\item $DN_A(\zeta)$, which is the projection of $\zeta$ on
  $\ker\op{L}_{N(A)}$ in the direction of $\fspace{C}$,
\item $\varphi_A=\varphi-\int\varphi\,\dd\mu_A \in\ker\mu_A$
\end{itemize}
and we apply Theorem \ref{theo:affine-derivative-normalized}
to this new potential, tangent vector, and test function.
We obtain exactly the desired expression once we notice
that
\begin{align*}
\lvert\mu_{A+\zeta}-\mu_{N(A)+DN_A(\zeta)}\rvert
  &= \lvert \mu_{N(A+\zeta)}-\mu_{N(A)+DN_A(\zeta)}\rvert \\
  &= O\big(\lVert N(A+\zeta) - N(A)-DN_A(\zeta) \rVert \big) \\
  &= O(\lVert \zeta\rVert^2).
\end{align*}
\end{proof}

\begin{proof}[Proof of Corollary \ref{coro:affine-derivative}]
We have to rewrite
\[ \int (I-\op{L}_{N(A)})^{-1}\varphi_A \cdot DN_A(\zeta) \,\dd\mu_A.\]
We first observe that the final expression we aim for only involves
$A$ through the measure $\mu_A$, so that we can as well replace $A$
by $N(A)$, i.e. assume that $A$ is normalized (this has for sole
purpose to avoid writing a dozen times $\op{L}_{N(A)}$).

We first write
$(I-\op{L}_A)^{-1}\varphi_A = \sum_{i\ge 0} \op{L}_A^i\varphi_A$,
and recall that $DN_A(\zeta)$ is the projection of $\zeta$ to
$\ker\op{L}_A$ along $\fspace{C}$; this means that there is a function
$g\in\fspace{X}(\Omega)$ such that
\begin{equation}
DN_A(\zeta) = \zeta_A + g-g\circ T
\label{eq:DNA}
\end{equation}
(where $\zeta_A = \zeta - \int \zeta \,\dd\mu_A \in \ker\mu_A$)
and that $\op{L}_A (DN_A(\zeta))=0$.
In particular, we have $\op{M}_A(DN_A(\zeta))=0$; thus,
\[g = \op{M}_A(DN_A(\zeta)-\zeta_A) = -\op{M}_A(\zeta)
  = \sum_{i\ge 1} \op{L}_A^i \zeta_A.\]

This leads us to
\begin{multline*}
\int (I-\op{L}_A)^{-1}\varphi_A \cdot DN_A(\zeta) \,\dd\mu_A \\
  = \int \sum_{i\ge 0} \op{L}_A^i \varphi_A \cdot \zeta_A \,\dd\mu_A
    + \int \sum_{i\ge 0}\op{L}_A^i \varphi_A \cdot g \,\dd\mu_A
  - \int \sum_{i\ge 0}\op{L}_A^i \varphi_A \cdot g\circ T \,\dd\mu_A \\
  = \int \sum_{i\ge 0} \op{L}_A^i \varphi_A \cdot \zeta_A\,\dd\mu_A
    + \int \sum_{i\ge 0}\op{L}_A^i \varphi_A \cdot g \,\dd\mu_A
   - \int \sum_{i\ge 0} \op{L}_A^{i+1} \varphi_A \cdot g \,\dd\mu_A \\
  = \int \sum_{i\ge 0} \op{L}_A^i \varphi_A \cdot \zeta_A\,\dd\mu_A
    + \int \varphi_A \cdot g \,\dd\mu_A \\
  = \int \sum_{i\ge 0} \op{L}_A^i \varphi_A \cdot \zeta_A\,\dd\mu_A
    + \int \varphi_A \cdot \sum_{i\ge 1} \op{L}_A^i \zeta_A \,\dd\mu_A,
\end{multline*}
and, using the invariance of $\mu_A$ under $\op{L}_A^\ast$:
\begin{multline*}
\int (I-\op{L}_A)^{-1}\varphi_A \cdot DN_A(\zeta) \,\dd\mu_A \\
  = \int \varphi_A \zeta_A \,\dd\mu_A
+ \sum_{i\ge 1}
    \int \big( \varphi_A \cdot \zeta_A\circ T^i + \varphi_A\circ T^i \cdot
    \zeta_A \big) \,\dd\mu_A,
\end{multline*}
where the sum converges (exponentially).

Finally, we observe that there is no use normalizing both
$\varphi$ and $\zeta$, since for example
$\int \varphi_A \zeta_A \,\dd\mu_A
= \int \varphi_A \zeta \,\dd\mu_A$. All $\zeta_A$
can therefore be replaced by $\zeta$, and we get the desired formula.
\end{proof}

\section{A Riemannian metric on the space of normalized potentials}
\label{sec:metric}

The goal of this section is to define and to study a (weak) Riemannian metric
on the space of Gibbs measures. More precisely, we construct
a Riemannian metric on the manifold of normalized potentials, which
corresponds equivalently to a Riemannian metric on the quotient space
$\fspace{Q}=\fspace{X}(\Omega)/\fspace{C}$, and relates in various ways
to dynamical quantities. After a conformal rescaling by the metric
entropy, this metric is very closely related to the metric defined by
McMullen \cite{Mac} (see also \cite{BCS} and references therein).

\subsection{Weak and strong inner products on Banach spaces}

Consider a positive symmetric bilinear form $\langle\cdot,\cdot\rangle$ on some
Banach space $\fspace{Y}$. There are two possible definitions
of positive-definiteness.
The first one is a copy and paste of the finite-dimensional definition,
that is, we ask that
\[\forall y \neq 0 \in\fspace{Y}:\quad \langle y,y\rangle >0.\]
In this case, one says that $\langle\cdot,\cdot\rangle$ is weakly
positive-definite.
The second one is to ask that the Banach norm $\lVert\cdot\rVert$
of $\fspace{Y}$ controls $\langle\cdot,\cdot\rangle$ from below, that is,
\[\exists C>0, \forall y\in\fspace{Y}: \quad
  \langle y,y\rangle \ge C\lVert y\rVert^2\]
In this case, one says that $\langle\cdot,\cdot\rangle$ is strongly
positive-definite; note that this condition implies
weak positive-definiteness.

Most of the time, one is only interested in bilinear forms
which are continuous with respect to the Banach topology of
$\fspace{Y}$. But if $\langle\cdot,\cdot\rangle$ is
both continuous and strongly positive-definite, then its
associated norm is equivalent to $\lVert\cdot\rVert$, and in particular
$\fspace{Y}$ must be isomorphic to a Hilbert space.
Therefore, most Banach spaces have no continuous,
strongly positive-definite inner product.

We shall say that $\langle\cdot,\cdot\rangle$ is an \emph{inner product}
if it is continuous and weakly positive-definite, and use
the term \emph{semi-definite inner product} for a merely continuous,
positive semi-definite symmetric bilinear form.
By a \emph{Riemannian metric} on a smooth Banach manifold,
we mean a field of inner products on the tangent spaces,
such that when translated in a chart, the inner product
depends smoothly on the point, that is, it defines a smooth map
from the domain of the chart to the Banach space of
symmetric bilinear forms.

As a last remark, note that when $\langle \cdot,\cdot\rangle$
is an inner product inducing a complete norm, it endows
$\fspace{Y}$ with a second structure of Banach space
(more precisely a Hilbert structure of course). Then the
the identity map $\fspace{Y}\to\fspace{Y}$ is a continuous
bijection between the two Banach structures at hand, and is therefore an
isomorphism. This implies in particular that $\langle \cdot,\cdot\rangle$
is strongly positive-definite.
In other words, inner products which are not strongly positive-definite
induce a norm which is never complete. This means that there will be
a relatively subtle interplay between the topology of $\fspace{Y}$
and the measurements made from $\langle \cdot,\cdot\rangle$.

\subsection{The Variance metric} \label{sec:Var}

Now, we introduce our proposed metric. Its main properties
are summed up in the following result.
\begin{theo}\label{theo:metric}
There exists an analytic map from $\fspace{X}(\Omega)$ to the
space of its continuous symmetric bilinear forms, which
maps any potential $A$ to a semi-definite inner product
$\langle\cdot ,\cdot\rangle_A$ such that:
\begin{enumerate}
\item $\langle\cdot ,\cdot\rangle_A$
restricts to $T_A\fspace{N}$ into an inner product for all
  $A\in\fspace{N}$, thus inducing a Riemannian metric on $\fspace{N}$,
\item this Riemannian metric coincides with
the one obtained from $L^2(\mu_A)$:
\[ \forall A\in\fspace{N}, \forall \eta,\zeta\in T_A\fspace{N}:\quad
  \langle\eta,\zeta\rangle_A = \int \eta \zeta \,\dd\mu_A,\]
\item for all $A$, $\langle\cdot ,\cdot\rangle_A$
  induces a well defined inner product on
  $\fspace{Q}$, thus inducing a Riemannian
  metric on this quotient space,
\item for all $A,\zeta,\varphi\in\fspace{X}(\Omega)$, it holds:
 \[\frac{\dd}{\dd t} \left.\int \varphi \,\dd\mu_{A+t\zeta}\right|_{t=0}
  = \langle\varphi,\zeta\rangle_A,\]
\item for all $A,\zeta\in \fspace{X}(\Omega)$, it holds:
  \[\Var(\zeta_A,\mu_A) := \lim\frac{1}{n} \int \Big(\sum_{i=0}^{n-1}
  \zeta_A\circ T^i \Big)^2 \,\dd\mu_A =  \langle\zeta,\zeta\rangle_A.\]
\end{enumerate}
\end{theo}

Of course, the metrics in $\fspace{N}$ and $\fspace{X}(\Omega)/\fspace{C}$
correspond one to the other through the natural identification
between these two spaces. There is really only one Riemannian metric, which
can be viewed in two ways. Any of the last two items completely specify
$\langle\cdot ,\cdot\rangle_A$, and can be taken as a definition.
Our point here is that these expressions define the same
bilinear form, inducing an inner product on $T_A\fspace{N}$.

The end of this section is devoted to the proof of Theorem
\ref{theo:metric}.
The first step leading to this result is to observe that
the expression in Corollary \ref{coro:affine-derivative}
is symmetric: in the right-hand side, $\zeta$ and $\varphi$
play the same role (up to normalization, but the formula holds
and was even proved with $\zeta$ also normalized to $\zeta_A$).
This means that the function against which $\mu_A$ is tested
and the direction in which $A$ is moved play precisely the same role
in the evolution of the integral, a somewhat surprising connection
(but which also follows from Corollary \ref{coro:eigendata}
and the Schwarz Lemma).
It also indicates that the right-hand side in Corollary
\ref{coro:affine-derivative} defines a symmetric bilinear form:
given $A\in\fspace{X}(\Omega)$, we define for all
$\eta,\zeta\in\fspace{X}(\Omega)$:
\begin{align*}
\langle \eta,\zeta \rangle_A
  &= \int \eta_A\cdot \zeta_A \,\dd\mu_A
  + \sum_{i\ge 1} \int \big(\eta_A\cdot \zeta_A\circ T^i
  + \eta_A \circ T^i \cdot \zeta_A \big) \,\dd\mu_A \\
  &= \int \eta_A\cdot \zeta_A \,\dd\mu_A
  + \sum_{i\ge 1} \int \big(\op{L}_{N(A)}^i(\eta_A)\cdot \zeta_A
  + \eta_A \cdot \op{L}_{N(A)}^i(\zeta_A) \big) \,\dd\mu_A
\end{align*}
The second expression shows that $\langle \eta,\zeta \rangle_A$
is a well defined number, and that it defines a continuous
symmetric bilinear form on $\fspace{X}(\Omega)$.
It also follows from Section \ref{sec:integral} that
$\langle \cdot,\cdot \rangle_A$ depends continuously on $A$;
in fact the analyticity of $\langle\cdot,\cdot\rangle_A$
follows from Corollary \ref{coro:eigendata} in the same
way as Corollary \ref{coro:analytic}:
we have $\langle\zeta,\eta\rangle_A=D^2(\log\Lambda)_A(\zeta,\eta)$
which depends analytically on $A$.

From Corollary  \ref{coro:affine-derivative}, we see that
\begin{equation}
\langle\eta,\zeta\rangle_A =
\int (I-\op{L}_{N(A)})^{-1}(\eta_A) \cdot DN_A(\zeta) \,\dd\mu_A,
\label{eq:asymmetric-inner}
\end{equation}
which seems asymmetric but will be useful.

As it is, $\langle \cdot,\cdot \rangle_A$ does not define an
inner product, because it is not weakly positive-definite.
\begin{prop}\label{prop:isotropy}
The symmetric form $\langle\cdot,\cdot,\rangle_A$ is
positive semi-definite, and for all $A\in\fspace{N}$ and
$\zeta\in T_A\fspace{N}$ we have
$\langle\zeta,\zeta\rangle_A= \int \zeta^2\,\dd\mu_A$.
Moreover given $A,\zeta \in\fspace{X}(\Omega)$, the following three
statements are equivalent:
\begin{enumerate}
\item\label{enumi:isotropy1} $\langle \zeta,\eta \rangle_A=0$, for all
  $\eta\in \fspace{X}(\Omega)$,
\item\label{enumi:isotropy2} $\langle \zeta,\zeta \rangle_A = 0$,
\item\label{enumi:isotropy3} $\zeta\in \fspace{C}$.
\end{enumerate}
\end{prop}

\begin{proof}
First, observe that when $A\in\fspace{N}$ and $\zeta\in T_A\fspace{N}$ we
have $\zeta_A=\zeta$ and $\op{L}_A(\zeta)=0$, so that:
\[\langle \zeta,\zeta \rangle_A
  = \int \zeta_A\cdot \zeta_A \,\dd\mu_A
  + \sum_{i\ge 1} \int \big(\op{L}_A^i(\zeta_A)\cdot \zeta_A
  + \zeta_A \cdot \op{L}_A^i(\zeta_A) \big) \,\dd\mu_A
  = \int \zeta\cdot \zeta \,\dd\mu_A
\]

It is clear that \ref{enumi:isotropy1} implies \ref{enumi:isotropy2},
and \eqref{eq:asymmetric-inner} shows that
\ref{enumi:isotropy3} implies \ref{enumi:isotropy1}. Let us
show that \ref{enumi:isotropy2} implies
\ref{enumi:isotropy3}. Since $\langle \cdot,\cdot\rangle_A$
does not change if we add a constant or a coboundary to $A$
(i.e. it only depends on $\mu_A$), we can assume that $A$ is normalized.

Suppose that $\zeta$ is isotropic, i.e.
$\langle \zeta,\zeta \rangle_A = 0$. By Proposition \ref{prop:LC}, we decompose $\zeta=\zeta'+f$, where $\zeta'\in\ker\op{L}_A$
and $f\in\fspace{C}$.
Then, $\langle \zeta',\zeta' \rangle_A =  \langle \zeta,\zeta \rangle_A$,
since $\fspace{C}$ is in the kernel of $\langle\cdot,\cdot\rangle_A$.
Thus,
we get $0=\langle \zeta',\zeta' \rangle_A=\int \zeta'^2 \,\dd\mu_A$.
It follows that $\zeta'=0$ and $\zeta=f\in\fspace{C}$.
\end{proof}
The last line of this proof is where we use \ref{hypo:support}.

\begin{defi}
Let $A\in\fspace{X}(\Omega)$ be any potential and
$[A]\in \fspace{Q}=\fspace{X}(\Omega)/\fspace{C}$ be its class modulo $\fspace{C}$.
For all $[\eta],[\zeta]\in \fspace{Q}$, we define
\[\langle[\eta],[\zeta]\rangle_{[A]} = \langle \eta,\zeta \rangle_A,\]
which is well-defined by Proposition \ref{prop:isotropy}, i.e. does not depend
on the chosen representatives in each class. If $A\in \fspace{N}$, we
still write $\langle \cdot , \cdot \rangle_A$ for the restriction of this inner product to $T_A\fspace{N}$. Proposition \ref{prop:isotropy} shows that both these products
are weakly positive-definite, and thus induce a norm on the Banach space
they are defined on ($\fspace{Q}$ and
$T_A\fspace{N} = \ker\op{L}_A$, respectively). We denote both norms
by $\lVert \cdot \rVert_A$, i.e.
\[\lVert \zeta\rVert_A = \sqrt{\langle \zeta,\zeta \rangle_A},\]
and we use this notation for general $\zeta\in\fspace{X}(\Omega)$.
\end{defi}


Let us now prove the last statement of Theorem
\ref{theo:metric}. As usual, we can define the variance
of a function in $\ker\mu_A$ by
\[\Var(\zeta_A,\mu_A) := \lim\frac{1}{n} \int \Big(\sum_{i=0}^{n-1}
  \zeta_A\circ T^i \Big)^2 \,\dd\mu_A.\]
By direct computation, we obtain
\begin{align*}
\int \Big(\sum_{i=0}^n \zeta_A\circ T^i \Big)^2 \,\dd\mu_A
  &= \sum_{i,j=0}^{n-1}
     \int \zeta_A\circ T^i \cdot \zeta_A\circ T^j \,\dd\mu_A \\
  &= n\int \zeta_A^2\,\dd\mu_A + 2\sum_{0\le i<j\le n-1} \int \zeta_A
    \cdot \zeta_A\circ T^{j-i} \,\dd\mu_A \\
  &= n\int \zeta_A^2\,\dd\mu_A + 2\sum_{k=1}^{n-1} (n-k)\int \zeta_A
    \cdot \zeta_A\circ T^k \,\dd\mu_A.
\end{align*}
Assume without loss of generality that $A$ is normalized. Then,
\begin{align*}
\frac{1}{n} \int \Big(\sum_{i=0}^{n-1} \zeta_A\circ T^i \Big)^2 \,\dd\mu_A
 &= \int  \zeta_A^2\,\dd\mu_A + 2\sum_{k=1}^{n-1} \int \zeta_A
    \cdot \zeta_A\circ T^k \,\dd\mu_A  \\
  &\qquad -\frac{2}{n} \sum_{k=1}^{n-1}
     \int k\op{L}_A^k(\zeta_A)\cdot\zeta_A \,\dd\mu_A,
\end{align*}
where the last term is bounded in norm by
$O(\frac{\lVert \zeta_A\rVert^2}{n} \sum_{k=1}^\infty k \delta^k)$
with $\delta$ the spectral gap of $\op{L}_A$.

It follows that for all $A,\zeta$ we have
\[
\Var(\zeta_A,\mu_A) = \lVert \zeta \rVert_A ^2
\]
which, as usual in common examples and as follows from Proposition
\ref{prop:isotropy}, vanishes exactly when $\zeta\in\fspace{C}$.

This concludes the proof of Theorem \ref{theo:metric}, and of Theorem
\ref{theomain:metric} once one observes that
\[D^2(\log\Lambda)_A = D G_A = \langle\cdot,\cdot\rangle_A.\]

\begin{rema}\label{rema:straight}
We have thus recovered in our setting the convexity of $\log\Lambda$,
which will be given -- as is customary-- an interpretation in term of ``pressure'' below.

A notable consequence of this is
that the submanifold $\fspace{N}$ does not
contain any straight interval: it is contained in the zero level set
of $\log\Lambda$, so that any straight interval in $\fspace{N}$
passing through $A$ would have its direction in
$T_A\fspace{N}$ and in the kernel of $\langle\cdot,\cdot\rangle_A$,
whose intersection is trivial.
\end{rema}

\section{Regularity of the Gibbs map: the Wasserstein structure} \label{sec:Was}

The development of optimal transportation and more
precisely of the $2$-Wasserstein distance has let an alternative
differential structure for the set $\spacem{P}(\Omega)$
emerge, notably driven by the work of Otto \cite{Otto}, Benamou and Brenier \cite{BB} and Ambrosio, Gigli and Savar\'e \cite{AGS}.
We shall rely on the formulation given by \cite{Gigli:structure}, which allows
to define the differentiability of a map at a point (as opposed to
more global notions, such as speed vectors defined almost everywhere).
One could in principle consider the case when $\Omega$ is a
Riemannian manifold, but for simplicity we shall restrict to $\Omega=\mathbb{S}^1=\mathbb{R}/\mathbb{Z}$ throughout this section.

\subsection{Elements of optimal transportation}

We will not give much details on optimal transportation, but many
references are available (e.g. \cite{Villani2} for a comprehensive
source). Let us say that the $2$-Wasserstein distance $W_2$
is a metric compatible with the weak topology, defined on
$\spacem{P}(\Omega)$ as the least cost
needed to move one measure to another, when the cost to move a unit of
mass is proportional to the \emph{squared} distance between the starting
point and the stopping point.

For each $\mu\in \spacem{P}(\mathbb{S}^1)$,
Gigli introduces a tangent space $T_\mu\spacem{P}(\mathbb{S}^1)$ which
may be only a metric cone, but turns out to be a Hilbert space in a number
of cases. There are several possible definitions of such a tangent space
(or cone), e.g. in term of geodesics, in term of measures on the tangent
bundle, or in term of vector fields on the manifold; the work of Gigli
ties all these points of view together when $\mu$ belongs to a certain class of
``nice'' measures.
In the present one-dimensional case, the relevant class
to be considered is the set of atomless measures. Assuming
$\mu\in\spacem{P}(\mathbb{S}^1)$ has no atom, one can consider as
tangent space to $\spacem{P}(\mathbb{S}^1)$ at $\mu$ the space
\[T_\mu\spacem{P}(\mathbb{S}^1) :=
 L^2_\nabla(\mu):= \overline{\{\nabla f \,|\,
 f\in C^\infty(\mathbb{S}^1,\mathbb{R}) \}}^{L^2(\mu)}\]
of vector fields on $\mathbb{S}^1$ which are square-integrable
with respect to $\mu$ and which are limits of gradients
of smooth function in $L^2(\mu)$ (note that the quotient
structure of $\mathbb{S}^1=\mathbb{R}/\mathbb{Z}$ makes it possible
to identify all tangent spaces of $\mathbb{S}^1$ with $\mathbb{R}$,
so that we can see vector fields as functions, and $\nabla f$ is
simply $f'$).
There is an obvious exponential map: given $\mu$ and $v\in
T_\mu\spacem{P}(\mathbb{S}^1)$ one sets
$\exp_\mu(v)= (\Id+v)_\#\mu$, i.e. the mass at any point
$x\in\mathbb{S}^1$ is moved to $x+v(x)\mod 1$.
Then for each
$v\in T_\mu\spacem{P}(\mathbb{S}^1)$,
one gets an exponential curve
$(\exp_\mu(tv))_{t\in [0,\varepsilon)}$ which has the property that
\[W_2(\mu,\exp_\mu(tv)) = t\lVert v\rVert_\mu + o(t)\]
where $\lVert v\rVert_\mu$ is the $L^2(\mu)$-norm of $v$
(here the fact that $v$ can be approximated by gradients is crucial).

We will say that a curve $t\mapsto\mu_t$ from an interval to
$\spacem{P}(\mathbb{S}^1)$ is Wasserstein-differentiable at $t_0$
with tangent vector $v\in T_{\mu_{t_0}} \spacem{P}(\mathbb{S}^1)$
whenever it holds
\[W_2(\mu_{t_0+h},\exp_{\mu_{t_0}}(hv)) = o(h).\]
Similarly, a map
$H:\fspace{Y}\to \spacem{P}(\mathbb{S}^1)$ from a Banach space to the set of
probability measures on $\mathbb{S}^1$ is
Wasserstein-differentiable at a point $A\in\fspace{Y}$ in a direction
$\zeta\in\fspace{Y}$ whenever there exist $v\in T_\mu\spacem{P}(\mathbb{S}^1)$
such that
\[W_2\Big(H(A+t\zeta),\exp_{H(A)}(tv)\Big) = o(t)\]
i.e. the tangent vector $v$ describes the first-order variations of $H$
in the Wassertein distance.
Of course, one can define more stringent versions of this definition
(Fr\'echet-like rather than G\^ateaux-like), but since our result is
negative we get the strongest statement by sticking to the weakest definition.

When $\Omega$ is a manifold, in each of its variations
(G\^ateaux or Fr\'echet), Wasserstein
differentiablity is stronger than the corresponding variation of
affine differentiability because of the \emph{continuity equation} below;
roughly, affine differentiability is about recording the ``vertical''
variations of the measure, i.e. the variation of weight it gives
to any given set, while Wasserstein differentiability is about
recording the ``horizontal'' variations of the measure, i.e. how one
should move the mass in the most economical way in order to obtain the
given change of measure. The physical principle of mass preservation
leads to the continuity equation, which in the present case
$\Omega=\mathbb{S}^1$ has the following form:
\begin{lemm}
Assume that $(\mu_t)_t$ is a curve of probability measures on $\mathbb{S}^1$
which is differentiable at $0$ with tangent vector
$v\in T_{\mu_0} \spacem{P}(\mathbb{S}^1)$, then for all smooth
function $\varphi$ we have
\[\frac{\dd}{\dd t} \int \varphi \,\dd\mu_t \Big|_{t=0}= \int \varphi' v \,\dd\mu_0.\]
\end{lemm}
The most common version of the continuity equation is stated
for curves of measures, with the above equality integrated over
time. The proof of the present version is very simple and
can be found in \cite{Kl2}.

A curve $(x_t)_{t\in I}$ in a metric space is said to be
\emph{absolutely continuous} whenever there is a positive function
$g\in L^1(I)$ such that for all $t_0,t_1\in I$:
\[d(x_{t_0},x_{t_1}) \le \int_{t_0}^{t_1} g(s) \,\dd s\]
(note when considering a curve $(\mu_t)$ in $\spacem{P}(\Omega)$,
that this notion as nothing to do with each
measure $\mu_t$ being absolutely continuous or not!)
In other words, an absolutely continuous curve is a curve whose speed
exists almost everywhere and is integrable. A particular case
is given by Lipschitz curves, whose speed
is in $L^\infty$; absolute continuity is therefore a very mild regularity
condition. A Rademacher theorem holds in this setting:
an absolutely continuous curve in $\spacem{P}(\mathbb{S}^1)$ endowed with
the $2$-Wasserstein distance is differentiable at almost every time
and satisfies the mean value theorem (see \cite{AGS}).

\subsection{Roughness of the Gibbs map in the Wasserstein space}

We are now in a position to state and prove the main
result of this section, which shows that the Gibbs map is
very far from being Wassertein-smooth.
\begin{theo}\label{theo:roughness}
Assume $T$ is $x\mapsto dx \mod 1$ acting on $\mathbb{S}^1$ and
$\fspace{X}(\mathbb{S}^1)$ is the space of $\alpha$-H\"older functions
for some $\alpha\in(0,1]$.
If $(A_t)_t$ is any smooth curve in $\fspace{X}(\mathbb{S}^1)$,
then its image curve $(\mu_{A_t})_t$ under the Gibbs map
is not (even locally) absolutely continuous in $(\spacem{P}(\mathbb{S}^1),W_2)$
unless it is constant (i.e. unless $A_t\in A_0+\fspace{C}$ for all
$t$).
\end{theo}

Recalling the interpretation of the Wasserstein metric $W_2$ above, we see
that changing smoothly the potential changes smoothly the levels of the
Gibbs measure (Theorem \ref{theo:affine-derivative}), but in a way that
corresponds to brutal reallocations of the mass distribution (Theorem
\ref{theo:roughness}). This result should be compared to
Corollary 1.3 in \cite{KLS}, where a Lipschitz-regularity result
is proved for the Gibbs map when $\spacem{P}(\Omega)$ is endowed
with the $1$-Wasserstein distance (which however does not yield
a differentiable structure).

The proof mostly relies on the following
point-wise non-differentiability result.
\begin{prop}\label{prop:roughness}
Under the same assumption as in Theorem \ref{theo:roughness},
consider the Gibbs map
$G:\Holder_\alpha(\mathbb{S}^1) \to\spacem{P}(\mathbb{S}^1)$
sending each $A$ to $\mu_A$.

If $G$ is Wasserstein-differentiable at any potential
$A$ in any direction $\zeta$, then either
$\mu_A$ is the Lebesgue measure (i.e. $A\in\fspace{C}$)
or the derivative vanish (i.e.
$W_2(\mu_{A+t\zeta},\mu_A) = o(t)$).
\end{prop}

\begin{proof}
Assume that $G$ is Wasserstein-differentiable at $A$ in the direction
$\zeta$.

If $\varphi$ is any smooth function, on the one hand the continuity
equation gives
\[\frac{\dd}{\dd t} \int \varphi \,\dd\mu_{A+t\zeta} \Big|_{t=0}
  = \int \varphi' v \,\dd\mu_A\]
where $v\in L^2(\mu_A)$ is some vector field (which can be approximated
by gradients in $L^2(\mu_A)$); on the other hand
section
\ref{sec:integral} gives
\[\frac{\dd}{\dd t} \int \varphi \,\dd\mu_{A+t\zeta} \Big|_{t=0}=
\int (I-\op{L}_A)^{-1}(\varphi_A) \cdot DN_A(\zeta) \,\dd\mu_A.\]

We get two very different-looking linear forms in $\varphi$ which both
describe the variations of its integral. The
proof will thus be complete as soon as we prove that unless
$\mu_A$ is the Lebesgue measure, these two forms can agree only by
vanishing.

For this, we use the following approximation lemma.
\begin{lemm}
Let $\mu$ be a measure on $\mathbb{S}^1$ which is singular with respect
to the Lebesgue measure and without atoms;
then for all $f\in L^2(\mu)$, and all $\beta<1$ there
is a sequence
of smooth functions $\varphi_n:\mathbb{S}^1\to \mathbb{R}$ such
that $\varphi_n' \to f$ in $L^2(\mu)$ and $\varphi_n \to 0$ in
$\Holder_\beta(\mathbb{S}^1)$.
\end{lemm}

\begin{proof}
We first claim that when $I\subset [0,1]$ is an interval of
length $\ell$,
$w:I\to \mathbb{R}$ is measurable and $\mu$-essentially bounded by some
number $M$, and $\varepsilon>0$, there is a smooth function
$\varphi:I\to\mathbb{R}$ such that $\varphi$ and all its derivatives
vanish at the endpoints of $I$,
$\lVert \varphi\rVert_\infty\le \varepsilon$,
$\lVert \varphi' \rVert_\infty \le M$, and
$\int_I (\varphi'-w)^2 \,\dd\mu \le \varepsilon^2\ell^2$.

Let $\eta>0$ be arbitrary, to be chosen later on.
Since $\mu$ is concentrated on a $\lambda$-negligible set,
there is a finite set of intervals $I_1,\dots,I_k\subset I$
with disjoint interiors whose total length is less than $\eta$ and whose
complement
$J = I\setminus (I_1\cup\dots\cup I_k)$
is given by $\mu$ a mass less than $\eta$. Let $w_1$ be
the function which:
\begin{itemize}
\item is constant on each $I_i$, with value the $\mu$-average
  of $w$ on $I_i$,
\item is constant on $J$, with value such that
  $\int_I w_1 \,\dd\lambda =0$.
\end{itemize}
By taking $\eta$ small enough and by dividing the intervals $I_i$
into smaller intervals, we can ensure that $\int (w- w_1)^2 \,\dd\mu$
is arbitrarily small.

Let $w_2$ be a smooth approximation of $w_1$ such that
$\int (w- w_2)^2 \,\dd\mu$ stays small, $w_2$ is bounded
by $M$, $\int_I w_2=0$, and $w_2$ is zero on some neighborhoods of the
endpoint of $I$ (this last condition is easy to fulfill since
$\mu$ has no atom).

Define a smooth, $M$-Lipschitz function $\varphi$ by
\[\varphi(x) := \int_a^x w_2(t) \,\dd t\]
where $a=\min I$ is the starting point of $I$.
Then $\varphi'=w_2$ is close to $w$ in $L^2(\mu,I)$ norm and
bounded above by $M$ (though $\varphi''$ is extremely large), and
$\varphi$ and its derivatives vanish at both endpoints of
$I$.
The uniform norm of $\varphi$ is then bounded by $M\eta$,
and can thus be made arbitrarily small, proving
the claim.

Now, given $v$ and an integer $n$, choose a
$\mu$-essentially bounded function $\bar v$ which is $1/n$-close
to $v$ in $L^2(\mu)$, call $M$ its essential bound,
then choose $\ell$ small enough to
ensure that $\ell^{1-\beta} M<1/n$. Divide $\mathbb{S}^1$ into
intervals of length $\ell$ and apply the claim to each of them.
The boundary conditions enable us to glue the smooth functions
defined on each interval into a smooth function $\varphi_n$
defined on $\mathbb{S}^1$, such that $\varphi_n'$
is $M$-bounded and $1/n$-close to $\bar v$ in $L^2(\mu)$
and $\lVert \varphi_n\rVert_\infty <1/n$.
For any $x,y\in\mathbb{S}^1$, when $|x-y|\le \ell$ we get
\[\frac{|\varphi_n(x)-\varphi_n(y)|}{|x-y|^\beta}
  \le \lVert\varphi_n'\rVert |x-y|^{1-\beta}
  \le M \ell^{1-\beta} \le \frac{1}{n}\]
and when $|x-y|\ge \ell$ we get
\[\frac{|\varphi_n(x)-\varphi_n(y)|}{|x-y|^\beta}
  \le  \frac{|\varphi_n(x)-0|+|0-\varphi_n(y)|}{|x-y|^\beta}
  \le \frac{2M\ell}{|x-y|^\beta}
  \le 2M \ell^{1-\beta} \le 2/n.\]

This proves the Lemma.
\end{proof}

Now we simply apply the Lemma to $f=v$, and $\beta=\alpha$
if $\alpha<1$, or any lower $\beta$ otherwise
(using that the thermodynamical formalism holds for the current
$T$ with any $\beta$). This gives
us smooth functions
$\varphi_n$ such that
\[\int v^2 \,\dd\mu_A =
\lim_n \frac{\dd}{\dd t} \int \varphi_n \,\dd\mu_{A+t\zeta} \Big|_{t=0}
  =  \int (I-\op{L}_A)^{-1}(0)
  \cdot DN_A(\zeta) \,\dd\mu_A=0\]
(every operator being interpreted in the $\beta$-H\"older space if necessary)
so that $v$ vanishes $\mu_A$-almost everywhere, and the Wasserstein
derivative of $\mu_{A+t\zeta}$ vanishes.
\end{proof}

\begin{proof}[Proof of Theorem \ref{theo:roughness}]
If $(\mu_{A_t})_t$ is absolutely continuous, it is
differentiable almost-every\-where and from Proposition
\ref{prop:roughness}
we deduce that at each $t$ such that $\mu_{A_t}$ is differentiable
and not Lebesgue, its derivative vanishes. The mean value inequality
then ensures that $(\mu_{A_t})$ must then be constant.
\end{proof}

We end this section with some open questions. First, Proposition
\ref{prop:roughness} leaves open the following.
\begin{ques}
In the case of $T:x\mapsto dx\mod 1$, is the Gibbs map
differentiable at $A$ when $A\in\fspace{C}$
(i.e. when $\mu_A$ is the Lebesgue measure)?
\end{ques}

Second, note that the analogue
of Theorem \ref{theo:roughness} for the shift is true
independently of the map $G$,
since the $2$-Wasserstein space of an ultrametric space
such as $\mathcal{A}^{\mathbb{N}}$ contains no absolutely
continuous curve at all (see \cite{Kl3}). But the $2$-Wasserstein space
of a manifold contains plenty of absolutely continuous curves
(it is even a geodesic space), so when $\Omega$ has a smooth
structure, the irregularity of $G$ with respect to the Wasserstein
metric can be
somewhat surprising. One then wonders how much it has to do with
$G$, and how much it has to do with its image:
\begin{ques}
Assume $\Omega$ is a manifold and $T$ is smooth. Are there
any non-constant, absolutely continuous curves
$(\mu_t)_t$ in $(\spacem{P}(\Omega),W_2)$ such that
$\mu_t$ is $T$-invariant for all $t$? What about
the subset of Gibbs measures with $\alpha$-H\"older potential?
\end{ques}
In other words, we ask whether the set of $T$-invariant measures
is a nice, somewhat smooth subset of the set of all probability
measures, or if from the Wasserstein point of view
it is a very irregular subset of $\spacem{P}(\mathbb{S}^1)$
(one can think of the Von Koch curve in $\mathbb{R}^2$ as an example
of a connected, very irregular subset of a smooth space).

\section{Application to equilibrium states}
\label{sec:applications}

In this section we use the differential calculus developed
above to study several classical optimization problems.

\subsection{Entropy and pressure}

Given our broad framework, we shall use the following
Legendre transform definition for entropy: for any
$T$-invariant measure $\nu$, we set
\[\entropy(\nu) := \inf_{A\in\fspace{X}} \big(\log \lambda_A
  -\int A \,\dd\nu \big).\]
Note that this quantity a priori depend on the chosen class
of function $\fspace{X}(\Omega)$; but in many cases it
is in fact equal to the metric entropy of $\mu$, see remarks \ref{rema:definition2} and \ref{rema:definition3}.
The assumption \ref{hypo:large} ensures that
$\fspace{X}(\Omega)$ is quite large, preventing
$\entropy{}$ to be too degenerate.
\begin{rema}
The number $\log\lambda_A-\int A\,\dd\nu$ only depend on the
class $[A]$ of $A$ modulo $\fspace{C}$ (adding a constant
to $A$ changes $\log\lambda_A$ and $\int A\,\dd\nu$ by
the same additive constant, and adding a coboundary leaves
both terms unchanged). In particular, one can rewrite
\[\entropy(\nu) = \inf_{A\in\fspace{N}} \int(-A) \,\dd\nu\]
and observe that $A(y)=\log e^{A(y)}$ where
\[\mathbb{P}(x\to y) := \begin{cases}
 e^{A(y)} &\mbox{when }T(y)=x \\
 0 &\mbox{otherwise} \end{cases} \]
defines transition probabilities for a
Markov chain on $\Omega$ supported on backward orbits of
$T$. In other words, $H_\nu$ is the infimum of
$\int \big(-\log \mathbb{P}(T(y)\to y)\big)\,\dd\nu(y)$
over Markov chains supported on backward orbits of $T$,
such that transition probabilities depends on the endpoint,
with a regularity specified by $\fspace{X}(\Omega)$.
%
\end{rema}

Together with such a definition of entropy naturally comes
a dual quantity, the \emph{pressure}: for any potential
$B\in\fspace{X}(\Omega)$ we set
\[\pressure(B) := \sup_{\mu\in\spacem{P}_T(\Omega)}
  \big( \entropy(\mu) + \int B\,\dd\mu \big).\]
In many cases (e.g. shift in the Bernoulli space),
this turns out to coincide with the classical topological pressure
(see again remarks \ref{rema:definition2} and \ref{rema:definition3}).
Here we will concentrate on the study of the above Legendrian
formulations for these quantities, as they fit our framework most
naturally.

One of our main concern is to understand when and where
the above infimum and supremum are attained; we thus consider
the families of functionals defined for
$\mu,\nu\in\spacem{P}_T(\Omega)$ and $A,B\in\fspace{X}(\Omega)$ by
\begin{align*}
H_\nu(A) &= \log \lambda_A  -\int A\,\dd\nu \\
P_B(\mu) &= \entropy(\mu)+\int B\,\dd\mu
\end{align*}
The functional $P_B$ is defined for all $T$-invariant measures
but we shall also study its restriction to Gibbs measures,
considered as acting on potentials:
\[P_B(A)= \entropy(\mu_A)+\int B\,\dd\mu_A.\]
We will abusively use the same name $P_B$
for the map defined on invariant measure, the map
defined on potentials, its restriction to normalized
potential and the map it induces on the quotient
$\fspace{Q} = \fspace{X}(\Omega)/\fspace{C}$. The way we write the
argument ($P_B(\mu_A)$, $P_B(A)$ or $P_B([A])$)
will usually make the difference clear.

Since $H_\nu$ is $\fspace{C}$-invariant,it induces a functional on the
quotient $\fspace{Q}$, which we still denote by $H_\nu$.

\subsection{Classical equilibrium states and Legendre duality}
\label{sec:variational}

\subsubsection{The entropy functionals}

We start with the study of the functionals $H_\nu$.

\begin{prop}\label{prop:diff-entropy}
For all $\nu\in\spacem{P}_T(\Omega)$, the functional
$H_\nu$ on $\fspace{X}(\Omega)$ is analytic with
\[D(H_\nu)_A(\zeta) = \int \zeta \,\dd\mu_A - \int \zeta \,\dd\nu.\]
Moreover the map $[A]\mapsto H_\nu([A])$ induced
on $\fspace{Q}$ is strictly convex.
\end{prop}

\begin{proof}
Let us recall that $\Lambda:\fspace{X}(\Omega) \to (0,+\infty)$
is the analytic functional defined by $\Lambda(A)=\lambda_A$, and that
for all $A,\zeta\in\fspace{X}(\Omega)$ we have
$D(\log\Lambda)_A(\zeta) = \int \zeta \,\dd\mu_A$
(Corollary \ref{coro:eigendata}).
Since the second term in $H_\nu(A)=\log\lambda_A - \int A \,\dd\nu$
is linear and thus analytic
and equal to its derivative at any point, $H_\nu$ is analytic
with
$D(H_\nu)_A(\zeta) = \int \zeta \,\dd\mu_A - \int \zeta \,\dd\nu$.
The second term is constant in $A$, and by the work of Sections
\ref{sec:integral} and \ref{sec:metric} the second derivative is given by
\[D^2(H_\nu)(\zeta,\eta)= DG_A(\eta)(\zeta) = \langle\eta,\zeta\rangle_A.\]
In other words, considering the  functional $H_\nu$ induced on $\fspace{Q}$
we have $D^2(H_\nu)=\langle\cdot,\cdot\rangle_{[A]}$
which is positive-definite, proving the strict convexity on $\fspace{Q}$.
\end{proof}
Note that we do not have \emph{uniform} convexity (even locally)
since the inner product is only weakly positive-definite (there
are directions $[\zeta]$ with fixed size $\lVert[\zeta]\rVert$
such that the ``convexity'' $\lVert[\zeta]\rVert_{[A]}$ is arbitrarily
small).
Of course, $H_\nu$ is only weakly convex on
$\fspace{X}(\Omega)$ since it is constant along each fiber
$A+\fspace{C}$.

Proposition \ref{prop:diff-entropy} now implies the following result.
\begin{coro}\label{coro:vp-entropy}
When $\nu=\mu_B$ for some $B\in\fspace{X}(\Omega)$, then
 $H_{\mu_B}([A])$ is uniquely minimized at $[A]=[B]$ and thus
\[\entropy(\mu_B)=\log\lambda_B-\int B \,\dd\mu_B=-\int N(B) \,\dd\mu_B.\]
When $\nu$ is not in the image of the Gibbs map,
$H_\nu$ does not reach its infimum.
\end{coro}
Note that a normalized $B$ is non-positive and non-zero
and has $\lambda_B=1$, so that $\entropy(\mu_B)> 0$
(use \ref{hypo:support} to get the strict inequality).

\begin{proof}
Hypothesis \ref{hypo:large} implies
that $\fspace{X}(\Omega)$ ``separates measures'' i.e.
\[\Big( \forall \zeta\in\fspace{X}(\Omega):
  \int \zeta \,\dd\mu=\int \zeta \,\dd\nu\Big) \implies
  \mu=\nu.\]
Using $D(H_\nu)_A(\zeta) = \int \zeta \,\dd(\mu_A - \nu)$
we see that when $\nu$ is not in the image of the Gibbs map $H_\nu$
has no critical point, hence no minimum;
and when $\nu=\mu_B$ the critical points of $H_{\mu_B}$
are exactly the potentials $A$ such that $\mu_A=\mu_B$.
Going down
to the quotient we get only one critical point $[B]$ and
the strict convexity implies that this critical point is the unique
minimizer.
\end{proof}

\begin{rema}\label{rema:gibbsuniqueness}
At first glance, it looks like we used that the Gibbs map
$G:A\mapsto \mu_A$ is one-to-one in this proof, while we were only
able to prove it in some cases in Remark \ref{rema:gibbstandard}.
But in fact, the above proof rather \emph{implies}
the injectivity of $G$, as  by strict convexity
for all $\nu$ it can exist at most one
critical point of $H_\nu$ on $\fspace{Q}$.
\end{rema}

\begin{rema}\label{rema:definition2}
When $\nu=\mu_B$ is a Gibbs measure we thus obtain
\[\entropy(\mu_B)=\int \big(-\log e^{N(B)(y)}\big) \,\dd\mu_B(y)\]
where $e^{N(B)}$ can be interpreted as a transition probability,
or as the Jacobian of ``$\dd\mu/\dd\mu\circ T$'' (Remark
 \ref{rema:gibbstandard}). This can be used for some  $(\Omega,T,\fspace{X}(\Omega))$ to show that $\entropy(\mu_B)$ is equal to the metric entropy $h(\mu_B$); in particular this is the case for the shift $\sigma$ acting on  the Bernoulli space $\mathcal{A}^{\mathbb{N}}$ with H\"older potentials (the Classical Thermodynamical Formalism in the sense of \cite{PP}). 

In this case $(\sigma,\mathcal{A}^{\mathbb{N}},\Holder_\alpha)$ the equality $\entropy(\nu)=h(\nu)$ extends to any invariant probability $\nu$. Indeed by Theorem 9.12 in \cite{Wal} for any $\sigma$-invariant probability $\nu$ on the Bernoulli space, the metric entropy $h(\nu)$ satisfies
 $$ h(\nu) = \inf_{A\in C^0  (\mathcal{A}^{\mathbb{N}})} \Big\{ P(A) -\int A \,\dd \nu\Big\}$$
where $P$ is the topological pressure. As topological pressure is a continuous function on the continuous  potential $A$ (see Theorem 9.7 in \cite{Wal}) and the set of H\"older functions is dense in $C^0(\mathcal{A}^{\mathbb{N}})$, the infimum above can be restricted to the H\"older potentials $A$. For H\"older potentials the pressure satisfies $P(A)=\log \lambda_A$ and this shows that  $\entropy(\nu)=h(\nu)$. Of course this reasoning applies to all cases
when the topological pressure coincides with $\log\Lambda$ and
the metric entropy is the Legendre dual of pressure.

 The analogous results is proved for Gibbs plans in Lemma 6 in \cite{LMMS2}. An invariant probability is  particular case of a Gibbs plan (see equation (1) in \cite{LMMS2}) and this provides another proof for the equality of entropies. 
\end{rema}

\subsubsection{The pressure functionals}

We are now in a position to extend the following classical
result to our general framework.
\begin{theo}[Gibbs measures are equilibirum states]\label{theo:variational}
For all $B\in\fspace{X}(\Omega)$, we have
$\pressure(\mu_B)=\log\lambda_B$ and
$\mu_B$ is the unique maximizer of
$P_B(\mu)=\entropy(\mu)+\int B\,\dd\mu$
among all $T$-invariant probability measures.
\end{theo}

\begin{proof}
Simply observe that
\[P_B(\mu)-\log\lambda_B = \inf_{A\in\fspace{X}(\Omega)} H_\mu(A)
  -H_\mu(B).\]
Consider the functional $A\mapsto H_\mu(A)-H_\mu(B)$: it
takes the value $0$ at $A=B$ and by
Corollary \ref{coro:vp-entropy} this is its infimum precisely
when $\mu=\mu_B$. We deduce that $P_B(\mu_B)=\log\lambda_B$
and that for any other measure $\mu\in\spacem{P}_T(\Omega)$,
$P_B(\mu)<\log\lambda_B$.
\end{proof}

\begin{rema}
The expression $\pressure(\mu_B)=\log\lambda_B$
shows that $\pressure{}$ and $\entropy{}$ are
really Legendre duals one to the other, since we can now write
the later
$\entropy(\mu)=\inf_A \pressure(A) -\int A\,\dd\mu$.
\end{rema}

\begin{rema}\label{rema:definition3}
We can deduce from that result that $\entropy{}$ is the metric entropy
and $\pressure{}$ the topological pressure whenever
we know the later to be equal to $\log\Lambda$ and the former to
be its Legendre dual. In particular, this holds when $T$ is the shift
over a finite alphabet and $\fspace{X}=\Holder_\alpha$, but of course
in this case it is
possible and more satisfactory to prove that $\entropy{}$ and $\pressure{}$
are the classical quantities\footnote{For example one can
proceed as in \cite{Lopes-etal}, noting that we use
here the classical normalization.} and recover their interpretation in terms
of eigenvalue and Legendre dual by the above.
\end{rema}

\begin{rema}
As a particular case, the measure of maximal entropy is
unique and equal to $\mu_0$ where $0$ is the zero of
$\fspace{X}(\Omega)$. One can then describe
$\mu_0$ as the stationary measure for the Markov chain
on $\Omega$ defined by the normalized potential $N(0)$.
When $T$ is $d$-to-one, then $-\log d$ is obviously
normalized and in the class of $0$ modulo $\fspace{C}$, so that
the measure of maximal entropy is the stationary measure
for the uniform random walk on backward orbits of $T$.

However, this hides some complications when points of
$\Omega$ do not
all have the same number of inverse images under the action of $T$:
it might then be quite difficult to express $N(0)$.
\end{rema}

\subsection{Gradients and gradient flows}

\subsubsection{Computation of some gradients}\label{sec:gradient}

We can now use the metric $\langle\cdot,\cdot\rangle_A$ to
define the gradients of the functionals $\entropy{}$ and $P_B$.

Note that a weak Riemannian metric such as $\langle\cdot,\cdot\rangle_A$
does not give a gradient to all $C^1$ functionals: indeed
$\langle\cdot,\cdot\rangle_A$ induces
a continuous, one-to-one map from the tangent space of $\fspace{N}$
to its dual, but this map is not onto.\footnote{If it where,
by Banach's isomorphism theorem the map $\zeta\mapsto\langle \zeta,\cdot\rangle_A$ from $\fspace{X}(\Omega)$ to its dual would be
an isomorphism,
which is equivalent to $\langle\cdot,\cdot\rangle_A$ being strongly
positive-definite.} Only those functional
whose differential belong to the image of this map will have a gradient.

First, the results
of Section \ref{sec:integral} and the very definition of
the metric yields that $G_\varphi:[A]\mapsto \int \varphi \,\dd\mu_A$
defined on the quotient $\fspace{Q}$
has a gradient:
\begin{align*}
D(G_\varphi)_A(\zeta) &= \langle \varphi,\zeta\rangle_A \\
  &= \langle [\varphi],[\zeta] \rangle_{[A]} \\
\nabla G_\varphi([A]) &= [\varphi].
\end{align*}
Similarly, the function $A\mapsto \int \varphi\,\dd\mu_A$ defined on
$\fspace{N}$ has a gradient at $A$, given by $DN_A(\varphi)$ (recall
that the gradient must be a vector in $T_A\fspace{N} = \ker\op{L}_A$
and that $DN_A$ is precisely the projection on this space
along $\fspace{C}$).

Then, we consider the map $A\mapsto \entropy(\mu_A)$.
As before, we will abusively denote by $\entropy{}$ this
map, as well
as its restriction to $\fspace{N}$ and the map it induces
on $\fspace{Q}$.

From $\entropy(A) =  -\int N(A) \,\dd\mu_A$,
the product rule yields
\begin{align*}
D(\entropy{})_A(\zeta) &= -\int DN_A(\zeta) \,\dd\mu_A
  - \langle N(A),\zeta \rangle_A \\
  &= \langle -A,\zeta \rangle_A
\end{align*}
since $DN_A(\zeta)\in\ker\op{L}_A \subset \ker\mu_A$ and
$\fspace{C}=\ker\langle\cdot,\cdot\rangle_A$.
This computation shows further that $\entropy{}$ (now considered
as induced on $\fspace{Q}$ or restricted to $\fspace{N}$)
has a gradient:
\[\nabla \entropy([A]) = -[A], \quad\mbox{or again}\quad
  \nabla \entropy(A) = -DN_A(A) \ \mbox{ when }A\in\fspace{N}.\]

Observing that $P_B(A) = \entropy(A) + G_B(A)$
we thus proved the following.
\begin{prop}\label{prop:gradients}
The maps $G_\varphi$, $\entropy{}$ and $P_B$
have gradients
for the weak Riemannian metric $\langle\cdot,\cdot\rangle_A$,
given by
\begin{align*}
\nabla G_\varphi([A]) &= [\varphi] &
                          \nabla G_\varphi(A) &= DN_A(\varphi)\\
\nabla \entropy([A]) &= -[A]  & \nabla \entropy(A) &= -DN_A(A) \\
\nabla (P_B) ([A]) &= [B-A] & \nabla (P_B)(A) &= DN_A(B-A)
\end{align*}
where the functionals are considered either on $\fspace{Q}$
(left column) or $\fspace{N}$ (right column).
\end{prop}

\subsubsection{Gradient flow}\label{sec:gradient-flow}

One particularly nice feature of the gradient of the pressure
$P_B$ computed in Section \ref{sec:gradient}  is that
it straightforwardly induces a gradient flow:
for all $[A_0]\in\fspace{Q}$,
there is a differentiable curve $[A_t]$ such that
for all $t$
\[\frac{\dd}{\dd t} [A_t] = \nabla(P_B)(A_t).\]
Indeed, a solution is given by
\[[A_t] = e^{-t}[A_0-B] + [B]. \]

Let us give a physical interpretation when $T$ is the shift:
we consider a system consisting of a $\mathbb{Z}$-lattice of particles,
a potential $A_0$ then represents a combination of the interaction
(and self-interaction) energy of the particles
and of the temperature, the Gibbs
measure $\mu_{A_0}$ is an equilibrium state (which minimizes
the ``free energy'' $-P_{A_0}$)
and represents the macroscopic state of the system at equilibrium.
Assume now that this system interactions changes instantly
to be now described by the potential $B$. The gradient flow above
is a natural and simple model for the evolution of the macroscopic state
of the system, where the systems evolves ``driven'' by $B$. Note
that in this interpretation, the state of the system out of equilibrium
is an equilibrium state for a varying potential.

\begin{rema}\label{rema:temperature}
Let us consider a particular case, where the interactions are
constant and only the temperature changes: $A_0 = \frac{1}{T_0} \varphi$
and $B= \frac{1}{T_1} \varphi$ for some $\varphi\in\fspace{X}(\Omega)$;
this corresponds to a system in contact with a heat bath whose temperature
changes suddenly. According to our model, the system then evolves only
 in its
temperature, as
\[ [A_t]=e^{-t}[A_0-B] + [B] = \Big(e^{-t}\big(\frac{1}{T_0}-\frac{1}{T_1}\big)+\frac{1}{T_1}\Big)[\varphi] \]
will be proportional to $[\varphi]$ for all $t$.
Note that here, $t$ should not be considered as the time
as the speed of evolution of temperature would not be right.
It might be possible to give a physical interpretation to
the parameter $t$, or to rescale the functional
$P_B$ and the metric in a way to obtain a physically
sound evolution of the temperature.
\end{rema}

\begin{rema}
Beware that this gradient flow really takes place on $\fspace{Q}$
(or equivalently, on $\fspace{N}$): it is not defined on the whole
of $\fspace{X}(\Omega)$ because there the metric has a non-trivial
kernel. Also, we cannot see this gradient flow as taking place in the
set of invariant measures with the Wasserstein structure,
because of Section \ref{sec:Was}:
the Gibbs map is not differentiable, and when $(A_t)$ is a integral
curve of our gradient flow, the curve $(\mu_{A_t})$ is not absolutely
continuous (Theorem \ref{theo:roughness})
and in particular not a gradient flow curve in the sense
of Ambrosio, Gigli and Savar\'e \cite{AGS}.
\end{rema}

\subsection{Prescribing integrals}\label{sec:prescribing}

In this section we study how one can find Gibbs measures
with prescribed values for the integrals of a given set of
test functions. This is both an application of the tools
we introduced here (in particular, the weak metric of
Section \ref{sec:metric} makes the proof quite easy), and a
main ingredient in the proof of existence and uniqueness
of equilibrium states under linear constraints.

Fix a tuple of test functions
$\Phi=(\varphi_1,\dots,\varphi_K)\in\fspace{X}(\Omega)^K$; we want to
study the set $\Rot(\Phi)$ of possible values taken by the
\emph{rotation vector}
\[\rv(\mu)=\big(\int \varphi_1\,\dd\mu, \dots, \int\varphi_K\,\dd\mu\big)
 \in\mathbb{R}^K \]
where $\mu$ runs over the set $\spacem{P}_T$ of $T$-invariant probability
measures, and the freedom one has to prescribe the values
of these integral with respect to a \emph{Gibbs} measure.

It is well-known and straightforward that $\Rot(\Phi)$ is convex;
it must also be bounded since potentials are assumed to be bounded by
\ref{hypo:Banach}.

Observe that if the classes modulo $\fspace{C}$
of the $\varphi_k$ are linearly
dependent, then their integrals with respect to any invariant measure
must satisfy a linear relation. Let us be more specific:
if $g-g\circ T+c$ is any element of $\fspace{C}$ and
$\mu$ is any $T$-invariant probability measure, then
$\int (g-g\circ T+c) \,\dd\mu = c$.
Therefore, if there is a non-trivial relation $\sum x_k [\varphi_k] = 0$
then there are $g\in\fspace{X}$ and $c\in\mathbb{R}$ such that
$\sum x_k \varphi_k = g-g\circ T + c$
and for all $\mu\in\spacem{P}_T$ we get the relation
$\sum x_k \int \varphi_k\,\dd\mu = c$,
constraining the vector of integrals
to an affine subspace of $\mathbb{R}^K$.
But this constraint on the rotation vector
can be worked out from the $\varphi_k$, and one can restrict to a maximal
subset of indexes
$S\subset \{1,\dots,K\}$ such that the family
$([\varphi_k])_{k\in S}$ is linearly independent.
Then the corresponding integrals will determine the integrals
of all $\varphi_k$. This procedure reduces the problem
to the case when the $[\varphi_k]$ are linearly independent,
which we will always assume in the sequel.

We then get the following (which does not pretend to much originality,
see \cite{KW1} and \cite{Jenkinson}; note that
our proof is close to the one by Kucherenko and Wolf,
but the metric $\langle\cdot,\cdot\rangle_A$ makes the injectivity
of the Jacobian obvious and we use a differential-geometric argument to
show that the map is onto).
\begin{theo}\label{theo:prescribing}
Let $\Phi=(\varphi_1,\dots,\varphi_K)\in\fspace{X}(\Omega)^K$ be such that the
classes $[\varphi_1],\dots,[\varphi_K]$ modulo $\fspace{C}$
are linearly independent.
Then for all $B\in \fspace{X}(\Omega)$, the map
\begin{align*}
\mathbb{R}^K &\to \interior \Rot(\Phi) \\
(a_1,\dots,a_K) &\mapsto \rv(\mu_{B+a_1\varphi_1+\dots+a_K\varphi_K})
\end{align*}
is an analytic diffeomorphism; in particular $\Rot(\Phi)$
has non-empty interior
and all its interior values are achieved by Gibbs measures.
\end{theo}

\begin{proof}
Consider the analytic maps
\begin{align*}
I : \mathbb{R}^K &\to \fspace{X}(\Omega) \\
  \bar\alpha = (\alpha_1,\dots,\alpha_K) &\mapsto B+\sum \alpha_k \varphi_k
\end{align*}
where $B$ is any fixed potential,
\begin{align*}
J : \fspace{X}(\Omega) &\to \mathbb{R}^K \\
  A &\mapsto \big(\int \varphi_1\,\dd\mu_A, \dots,
     \int\varphi_K\,\dd\mu_A\big)
\end{align*}
and their composition $L = J\circ I :\mathbb{R}^K\to \mathbb{R}^K$.
We also denote by $L_k$ the $k$-th component of $L$, i.e.
$L_k(\bar\alpha) = \int\varphi_k \,\dd\mu_{I(\bar\alpha)}$.

The differential of
$L$ is given by Sections \ref{sec:integral} and \ref{sec:metric}:
\[\frac{\partial L_k}{\partial x_j}(\bar\alpha)
  = \langle[\varphi_k],[\varphi_j]\rangle_{I(\bar\alpha)}.\]
This defines a Gram matrix, which is invertible
since the $[\varphi_k]$ are linearly
independent; it follows from the local inverse function theorem that
$L$ is a local diffeomorphism.

If $B$ is any potential, this implies that
$L(B)$ is in the interior of the image of $L$, in particular
in the interior of $\Rot(\Phi)$ (which must thus be non-empty).

What we have left to prove is
that $L$ is a \emph{global} diffeomorphism
from $\mathbb{R}^K$ to $\interior\Rot(\Phi)$.
Since that interior is diffeomorphic to $\mathbb{R}^K$, a Theorem of
\cite{Go}  reduces this to prove that $L$ is proper when its codomain
is taken to be $ \interior\Rot(\Phi)$, i.e. that whenever
a sequence $\bar x^{(n)}$ escapes compacts of $\mathbb{R}^K$, the points
$L(\bar x^{(n)})$ escapes the compacts of
$\interior\Rot(\Phi)$.
In other words, we want to prove that if
$\bar x^{(n)}\to\infty$ and $L(\bar x^{(n)})$ converges, the limit
lies on $\partial \Rot(\Phi)$.

Now, if $\bar x^{(n)}\to \infty$ and $L(\bar x^{(n)})$ converges,
up to taking a subsequence
we can assume that $\bar x_{(n)} = t_n \bar u + o(t_n)$
where $(t_n)$ is a diverging sequence of positive numbers,
and $\bar u$ is a unit vector in $\mathbb{R}^K$ (this is simply the
compactness of the unit sphere).

Observe that if $\bar x$ is a boundary point of $\Rot(\Phi)$ and
$\Phi= \sum y_k e_k^*$ (where $(e_k^*)$ is the canonical
dual basis)
is a linear form of $\mathbb{R}^K$ whose maximum on $\Rot(\Phi)$
is reached at $\bar x$, then
  \[\Phi(\bar x) = \max\Big\{\int \sum y_k \varphi_k \,\dd\mu
    \,\Big|\, \mu\in \spacem{P}_T \Big\},\]
and reciprocally points maximizing a linear form must lie on
the boundary.

Back to $L(\bar x^{(n)})$, we have
$I(\bar x^{(n)}) = t_n \varphi_{\bar u} + o(t_n)$
 where $\varphi_{\bar u} = \sum u_k \varphi_k$.
The variational principle tells us that
$\mu_{I(\bar x^{(n)})}$ maximizes
$\entropy(\mu)+\int (t_n\varphi_{\bar u} +o(t_n)) \,\dd\mu$
and it follows that the accumulation points of
this sequence of measures are all maximizing measures
of $\varphi_{\bar u}$. This precisely means
that the limit of $L(\bar x^{(n)})$ is a boundary point,
and we are done.
\end{proof}

As a by-product of this result, we get the following.
\begin{coro}\label{coro:weakdensity}
If $\fspace{X}(\Omega)$ is separable,\footnote{Or more generally
if in \ref{hypo:large} the approximation can be obtained
from a fixed countable subset of $\fspace{X}$}
then the set of Gibbs measures
$G(\fspace{X}(\Omega))$ is weakly dense in $\spacem{P}_T(\Omega)$.
\end{coro}

\begin{proof}
By assumption, there is a sequence $(\varphi_k)_{k\in\mathbb{N}}$
of elements of $\fspace{X}(\Omega)$ such that all continuous
$f:\Omega\to\mathbb{R}$ is the uniform limit of a subsequence
$(\varphi_{k_i})_{i\in\mathbb{N}}$.

Let $\mu\in\spacem{P}_T(\Omega)$;
from Theorem \ref{theo:prescribing}, for each
$K\in\mathbb{N}$ there is a potential $A_K\in\fspace{X}(\Omega)$
such that
\[\Big\lvert \int\varphi_k \,\dd\mu_{A_K} - \int \varphi_k \,\dd\mu
\Big\rvert <\frac{1}{K} \quad \forall k\in\{1,\dots,K\}.\]

Given any continuous $f:\Omega\to\mathbb{R}$ and any $\varepsilon>0$,
there is some $k_0$ such that
$\lVert f-\varphi_{k_0}\rVert_\infty \le \varepsilon$. For all
$K\ge \max(k_0,\frac{1}{\varepsilon})$ we thus have
\begin{align*}
\Big\lvert \int f \,\dd\mu_{A_K} - \int f\,\dd\mu \Big\rvert
  &\le \Big\lvert \int f \,\dd\mu_{A_K} - \int \varphi_{k_0}\,\dd\mu_{A_K}\Big\rvert
   + \Big\lvert \int \varphi_{k_0} \,\dd\mu_{A_K}
      - \int \varphi_{k_0}\,\dd\mu \Big\rvert\\
 &\qquad
 + \Big\lvert \int \varphi_{k_0} \,\dd\mu - \int f\,\dd\mu \Big\rvert \\
   &< 3\varepsilon
\end{align*}
Letting $\varepsilon\to0$, we see that $\int f\,\dd\mu_{A_K}\to \int f\,\dd\mu$, so that
$(\mu_{A_K})$ converges weakly to $\mu$.
\end{proof}

\subsection{Optimization under constraints}

Our goal here is to optimize the $P_B$ functionals
(for example, the entropy $\entropy{}$)
on natural subsets of
invariant measures, obtained by constraining the integrals
of some functions.
These questions have been considered by Jenkinson \cite{Jenkinson}
in the case of entropy and Kucherenko and Wolf
\cite{KW1,KW2}, with somewhat different assumptions and
methods. We believe that part of our claims are more explicit in some issues.

We fix as before test functions
$\Phi=(\varphi_1,\dots,\varphi_K)\in\fspace{X}(\Omega)^K$ and
we consider the set $\spacem{P}_T[\Phi]$
of $T$-invariant measures $\mu$ such that
$\int \varphi_k \,\dd\mu = 0$ for all $k$; among them are the Gibbs
measures whose normalized potential lies in
\[\fspace{N}[\Phi] := \big\{ A\in \fspace{N}
  \,\big|\, \forall k: \int\varphi_k \,\dd\mu_A = 0 \big\}\]
We will also denote by $\fspace{Q}[\Phi]$
the set of classes $[A]\in\fspace{Q}=\fspace{X}(\Omega)/\fspace{C}$
such that $A\in \fspace{N}[\Phi]$.

With these notation, we will prove the following constrained (or
``localized'') version of the variational principle.
\begin{theo}\label{theo:constraints}
Let $\Phi=(\varphi_1,\dots,\varphi_K), B\in\fspace{X}(\Omega)^K$ be
such that the $[\varphi_k]$ are linearly independent, and
such that $0$ is an interior vector of $\Rot(\Phi)$.
For each $B\in\fspace{X}(\Omega)$ denote by $B_0$ the
unique element $B_0=B+a_1\varphi_1+\dots+a_K\varphi_K$
such that $[B_0]\in \fspace{Q}[\Phi]$
(Theorem \ref{theo:prescribing}).

Then $\mu_{B_0}$ uniquely maximizes
$P_B$ over $\spacem{P}_T[\Phi]$, and the
value of the maximum is
$P_B(B_0)=\log\lambda_{B_0}$.
\end{theo}

\begin{proof}
We simply observe that for all
$\mu\in\spacem{P}_T[\Phi]$ we have
\[P_B(\mu) = \entropy(\mu) + \int (B_0-a_1\varphi_1-\dots-a_K\varphi_K)
     \,\dd\mu = \entropy(\mu) + \int B_0 \,\dd\mu = P_{B_0}(\mu).
\]
Applying Theorem \ref{theo:variational} to $P_{B_0}$ we see that
$P_B(\mu_{B_0})=P_{B_0}(\mu_{B_0})=\log \lambda_{B_0}$
is greater than $P_B(\mu)=P_{B_0}(\mu)$ whenever $\mu\neq \mu_{B_0}$
is in $\spacem{P}_T[\Phi]$.
\end{proof}

We can use this to recover in our setting another result from \cite{KW1}.
\begin{coro}\label{coro:KW1}
Let $\Phi=(\varphi_1,\dots,\varphi_K)\in\fspace{X}(\Omega)$,
such that the $[\varphi_k]$ are linearly independent and $B\in\fspace{X}(\Omega)$, and for $w\in\interior \Rot(\Phi)$ define
\[H(w)=\sup\{\entropy(\mu) ; \rv(\mu)=w\}.\]
Then $H$ is a positive, analytic map.
\end{coro}

\begin{proof}
By Theorems \ref{theo:prescribing} we know that there are
uniquely defined analytic functions
$a_k:\interior C\to\mathbb{R}$ such that
\[\rv\big(\mu_{a_1(w)\varphi_1+\dots+a_K(w)\varphi_K}\big)=w\quad\forall w\]
Setting $A(w)=a_1(w)\varphi_1+\dots+a_K(w)\varphi_K$ and
applying Theorem \ref{theo:constraints} to
$(\varphi_1-w_1,\dots, \varphi_K-w_K)$ we obtain
\[H(w) = \entropy(\mu_{A(w)}) =
\log\Lambda(A(w))-a_1(w)w_1-\dots-a_K(w)w_K,\]
proving the claim.
\end{proof}

\begin{rema}\label{rema:constraints}
Assume that $T$ is the shift over a finite alphabet and
$\fspace{X}=\Holder_\alpha$ (recall that
$h(\mu)= \entropy (\mu)$ in this case, Remark \ref{rema:definition2}).
Let $n$ be any positive integer,
and let $\Phi=(\varphi_1,\dots,\varphi_K)$
and $B$
be H\"older functions that only depends on the first $n$ coordinates, and
such that $\fspace{Q}[\Phi]$ is
non empty.

Then we claim that there is a unique measure maximizing
$P_B(\mu)$ among all elements of
$\spacem{P}_T[\Phi]$, and that this measure is a $(n-1)$-steps Markov measure
(i.e. a Gibbs measure $\mu_A$ such that $N(A)$
only depends on the first $n$ coordinates).
In particular, applying this to $B=0$, there is a $(n-1)$-steps
Markov measure
maximizing the entropy subject to any finite set of
simultaneously satisfiable constraints
$\int\varphi_k\,\dd\mu = 0$
whenever the $\varphi_k$ are constant on cylinder of depth $n$.

\begin{proof}
The only point that does not follow immediately from Theorem
\ref{theo:constraints} is that $\mu_A$ is $n$-Markov. But we know
that we can take $A=B+\sum x_k \varphi_k$ for some $(x_k)$;
notice that this $A$ might not be normalized, but is constant
on each depth-$n$ cylinder.

Now, $\op{L}_A$ preserves the subspace of $\fspace{X}(\Omega)$
made of functions that only depend on the first $(n-1)$
coordinates. In particular, for all $N$ the function
$\op{L}_A^N(\mathbf{1})$ only depends on the first $(n-1)$
coordinates. Since this is a closed space, the leading eigenfunction
$h_A$ only depends on the first $(n-1)$ coordinates, and
$h_A\circ T$ only depends on the first $n$ coordinates.

Now $N(A) = A + \log h_A - \log h_A\circ T -\log \lambda_A$
only depend on the first $n$ coordinates, which precisely means that
$\mu_A$ is $(n-1)$-steps Markov.
\end{proof}
\end{rema}

Let us give a couple of examples, which we will not make as general
as possible
but we will intentionally keep very explicit.
Let $\Omega=\{0,1\}^{\mathbb{N}}$, $T$ be the shift and
$\fspace{X}(\Omega)$ be a space of H\"older functions for
one of the usual metrics of $\Omega$. Given any
finite word $\omega$, let $\omega*$ be the cylinder
defined by $\omega$, i.e. the set of words starting with $\omega$.

\begin{exem}
Among shift-invariant measures $\mu$ such that
$\mu(0*)=.9$, the Bernoulli measure of parameter $.9$
(i.e. the law of the word $\alpha_1 \alpha_2 \dots$
where the $\alpha_j$ are i.i.d. random variables
taking the value $0$ with probability $.9$)
maximizes entropy.

Indeed, from Remark \ref{rema:constraints} we know that
there is a Bernoulli measure realizing this maximum, and
the Bernoulli measure with parameter $.9$ is the only one
to satisfy the constraint.
\end{exem}

\begin{exem}
Among shift-invariant measures $\mu$ such that
$\mu(01*) = 2\mu(11*)$, the Markov measure associated
to the transition probabilities
\begin{align*}
\mathbb{P}(0\to 0) &= 1-a & \mathbb{P}(0\to1) &= a \\
\mathbb{P}(1\to0) &= \frac23 & \mathbb{P}(1\to1) &= \frac13
\end{align*}
where $a$ is the only real solution to
\[(1-a)^5=\frac{4}{27} a^2 \qquad (a\simeq 0.487803)\]
maximizes entropy.

It is easily seen that the constraint is satisfiable by
a Markov measure, in particular by a Gibbs measure, thus we can
apply Remark \ref{rema:constraints}
to $B=0$, $K=1$ and $\varphi=\mathbf{1}_{10*}-2\cdot \mathbf{1}_{11*}$
where $\mathbf{1}_S$ is the indicator function of the set $S$.

The constraints easily translates into
$\mathbb{P}(1\to0) = \frac23$, and we define
$a= \mathbb{P}(0\to1)$. We know that the Gibbs entropy maximizing
measure is given by a potential of the form $A= x \varphi$ where
$x\in\mathbb{R}$; to translate this into the transition probabilities,
we only have to normalize $A$:
\[N(A) = x\varphi + \log h - \log h\circ T + \log \lambda\]
where $\lambda\in \mathbb{R}$ and $h$ only depends on the first coordinates
and matters only up to a multiplicative constant; we thus define
$\alpha = h(0*)/h(1*)$. Letting $\eta=e^x$, we then recover the transitions
probabilities as follows:
\begin{align*}
\mathbb{P}(0\to 0) &= e^{N(A)(00*)} = \lambda \\
\mathbb{P}(0\to 1) &= e^{N(A)(10*)} = \eta \alpha^{-1} \lambda \\
\mathbb{P}(1\to 0) &= e^{N(A)(01*)} = \alpha \lambda \\
\mathbb{P}(1\to 1) &= e^{N(A)(11*)} = \eta^{-2} \lambda
\end{align*}
We then have to solve the system
\[\left\{ \begin{array}{rcl}
1-a  &= & \lambda \\
a   &= & \eta \alpha^{-1} \lambda  \\
2/3 &= &\alpha\lambda \\
1/3 &= &\eta^{-2} \lambda
\end{array}\right.\]
This will give the only $\eta$ such that $\mu_A$ with the above $A$
satisfies the constraint, and from Remark \ref{rema:constraints}
we know that $\mu_A$ maximizes entropy under this constraint;
then the corresponding value of $a$ gives the transition probability
we seek. Note that, while we have some computation to do,
we do not have to estimate the actual entropy of Markov measures, nor
do we have to compute directly the eigendata of $\op{L}_A$.

The above system is easily solved by substitution:
$\lambda = 1-a$, then $\alpha=2/(3(1-a))$,
$\eta = 2a/(3(1-a)^2)$ and finally the last equation yields
$\big[2a/(3(1-a)^2)\big]^2 = 3(1-a)$, so that
$(1-a)^5=\frac{4}{27} a^2$.
\end{exem}

\section{Explicit computations for a restricted model}
\label{sec:2symbols}

In this section we explicitly show an example of the construction of section \ref{sec:metric} and some of its consequences. The dynamic we consider is the shift acting on the space $\{1,2\}^\mathbb{N}$.  We choose $\fspace{X}$ to be the space of $\alpha$-Holder functions for any $\alpha$, and denote by
$\fspace{X}_2$ the subset of potentials which depend only on the first two coordinates (of elements in $\{1,2\}^\mathbb{N}$). Note that we formally
cannot take $\fspace{X}_2$ as our full space of potentials, since it is not
invariant under composition by $T$.
It is easy to check that, in this setting,  \ref{hypo:Banach}-\ref{hypo:large} are satisfied (or one can find all the details in \cite{PP}).

\subsection{A positively curved metric}

If a potential $A \in \fspace{X}_2$ depends just on two coordinates then we can write $A(i,j)$ for the value of $A$ evaluated on the cylinder $ij\ast$ (i.e. the elements of
$\{1,2\}^\mathbb{N}$ of the type $ij\cdots$),
and we shall identify $\fspace{X}_2$ with the space of $2$ by $2$
real matrices.
The value of $e^A$ is well defined on such a cylinder,
and the action of the operator $\op{L}_A$
on potentials $\varphi$ depending
only on the first coordinate explicitly reads
\[ (\op{L}_{A}\phi(1\ast), \op{L}_{A}\phi(2\ast) )  =
\begin{pmatrix} \phi(1\ast) & \phi(2\ast) \end{pmatrix}
\begin{pmatrix}
           e^{ A_{11}}  & e^{ A_{12}} \\ e^{ A_{21}} & e^{ A_{22}}
\end{pmatrix}
\]
We can thus think of the operator $\op{L}_{A}$ as acting on a function as a left multiplication of the matrix, and we shall denote by $L$
the map
\[L : A=\begin{pmatrix} A_{11} & A_{12} \\ A_{21} & A_{22} \end{pmatrix}
  \mapsto L_A = \begin{pmatrix} e^{ A_{11}}  & e^{ A_{12}} \\
  e^{ A_{21}} & e^{ A_{22}} \end{pmatrix}\]
which exponentiate each coordinates of a matrix $A$, and
identify freely $\op{L}_A$ and $L_A$.

To normalize
the potential, that is to find the potential $\overline{A}:=N(A)$
differing from $A$ by a coboundary and a constant
such that $\op{L}_{\overline{A}}(1)= 1 $, we can apply the Perron-Frobenius theorem\footnote{see for example \cite{Gant} for the exact statement.}  to the matrix $L_A$
and solve with respect to the maximal eigenvalue and the left eigenvector, i.e. $ \ell L_A = \lambda_A \ell$. After the normalization, we obtain
\[
 \overline{A}(i,j) = \frac{e^{A(i,j) }\ell_i}{\lambda_A \ell_j}
\]
From now on, we will assume that $A$ is normalized and avoid the notation
$\overline{A}$.

We observe that the set $N(\fspace{X}_2) =: \fspace{N}_2$
of normalized potentials depending on two coordinates is defined by the equations
\[
 \left\{ \begin{array}{c} e^{ A_{11}} + e^{ A_{21}} = 1 \\ e^{ A_{12}} + e^{ A_{22}} = 1 \end{array} \right.
\]
so that $L(\fspace{N}_2)$ is the set of $2$ by $2$ column stochastic matrix,
denoted by $\fspace{S}_2$.

To sum up, a normalized potential in $\fspace{N}_2$ can be represented by
the matrix of its values on cylinders, subject to a nonlinear system
of constraints, or as a column stochastic matrix after coordinate-wise exponentiation. We thus obtain a natural chart $ S :  [0,1]\times [0,1] \to \fspace{S}_2$ by setting
\[
S(x,y) =
\begin{pmatrix}
x & 1 - y \\ 1- x & y
 \end{pmatrix}.
 \]
where $x,y \in (0,1)$ can be thought of
as transition probabilities $\mathbb{P}[1  \to 1]$ and
$\mathbb{P}[2  \to 2]$, respectively.

This parametrization has the advantage that
$\fspace{S}_2$ is (an open set of) an affine
subspace of $M_{2,2}(\mathbb{R})$: it has the same tangent
space at each point, a basis of which is given by
\[
 \frac{\partial S }{\partial x} = \begin{pmatrix}
1 & 0 \\ -1  &  0
 \end{pmatrix}
\qquad
 \frac{\partial S }{\partial y} = \begin{pmatrix}
0  & -1 \\ 0 & 1
 \end{pmatrix}.
\]
A tangent vector $\psi$ to $\fspace{S}_2$ at $S(x,y)$ shall be
written as $(\psi_1,\psi_2)$ in this basis, so that the corresponding
parametrized line $\gamma$
can be expressed for $s \in \mathbb{R}$ sufficiently small by
\[
\gamma_{S(x,y),\psi}(s) =
\begin{pmatrix}
x + s\psi_{1} & 1 - y - s\psi_{2} \\ 1- x - s\psi_{1} & y + s\psi_{2}
 \end{pmatrix} \in\fspace{S}_2.
 \]

The expression above is very readable, though does not allow us to compute right away the metric of $T_{A} \fspace{N}_2$. However, given $A \in \fspace{N}_2$ and  $\xi \in T_{A} \op{N}_2$, by Item $2$ of Theorem \ref{theo:metric}, we have that
\[
 \langle \zeta, \zeta \rangle_{A} = \int \zeta^2 \,\dd\mu_{A}
\]
It will thus be convenient to work both in $\fspace{N}_2$ where
the functional interpretation of matrices and vectors is clear,
and in $\fspace{S}_2$ where the Gibbs measures naturally appear.
If we consider a variation $\exp(A_{ij} + s\zeta_{ij})$ and differentiate at zero we obtain that the system
\[
 \left\{ \begin{array}{c} e^{ A_{11}}\zeta_{11} + e^{ A_{21}}\zeta_{21} = 0 \\ e^{ A_{12}}\zeta_{12} + e^{ A_{22}}\zeta_{22} = 0 \end{array} \right.
\]
defines $T_{A}\fspace{N}_2\subset \fspace{X}_2$
(in particular we see that this tangent plane depends
on the point $A$).
If $L_{A}=S(x,y)$ and $\psi$ corresponds to $\zeta$ in $T_{L_{A}}\fspace{S}_2$,
i.e. $\psi=DL_{A}(\zeta)$, it comes
\[
 \begin{pmatrix}
  \zeta_{11} & \zeta_{12} \\ \zeta_{21} & \zeta_{22}
 \end{pmatrix} =
 \begin{pmatrix}
  \frac{\psi_{1}}{x} &  \frac{-\psi_{2}}{1 - y} \\  \frac{-\psi_{1}}{1 - x} &  \frac{\psi_{2}}{y}
 \end{pmatrix}
\]
Now the matrix $S(x,y)$ has a right eigenvector
\[
\pi = (\pi(1\ast),\pi(2\ast)) = \bigg( \frac{1-y}{2-x-y}, \frac{1-x}{2-x-y} \bigg),
\]
which is the invariant measure on $\{1,2\}$ of the Markov chain defined
by $A$, and the measures of cylinders with respect to $\mu_A$
are
\[
 \begin{array}{c}
 \mu(11\ast) = \mathbb{P}[1  \to 1]\pi (1) \\
  \mu(12\ast) = \mathbb{P}[2  \to 1]\pi (2) \\
   \mu(21\ast) = \mathbb{P}[1  \to 2]\pi (1) \\
    \mu(22\ast) = \mathbb{P}[2  \to 2]\pi (2) \\
 \end{array}
\]
It is now easy to compute the metric:
\begin{equation*} \begin{split}
 \int \zeta^2 d\mu_{A}  & = \sum_{i,j} \zeta_{i,j}^2 \mu(ij\ast) = \frac{\psi_{1}^2}{x^2} \frac{x(1-y)}{2-x-y} + \frac{\psi_{2}^2}{(1 - y)^2} \frac{(1 - y)(1-x)}{2-x-y} \\
 & + \frac{\psi_{1}^2}{(1 - x)^2}\frac{(1- x)(1-y)}{2-x-y}  +   \frac{\psi_{2}^2}{y^2} \frac{y(1-x)}{2-x-y} \\
 & = \frac{1}{2-x-y}\left( \frac{1- y}{x(1- x)} \psi_{1}^2 + \frac{1- x}{y (1- y)} \psi_{2}^2 \right)
\end{split} \end{equation*}

\begin{prop}\label{metricprop}
The restriction $g$ of the variance metric $\langle\cdot,\cdot\rangle_{A}$
to $\fspace{N}_2$ is given in the chart $S$ by
\begin{equation}\label{MetricBernoulliSpace}
  g_{A} = \begin{pmatrix}
       \frac{(1-y)}{x(1-x)(2-x-y)} & 0 \\
        0 & \frac{(1-x)}{y(1-y)(2-x-y)} \\
        \end{pmatrix}
\end{equation}
(which is positive-definite for all $x,y\in (0,1)\times (0,1)$).
\end{prop}

This means that for $L_{A}=S(x,y)$ and $\psi=DL_{A}(\zeta)$ we have
$|\zeta|_{A}^2 = \begin{pmatrix}\psi_{1} & \psi_{2} \end{pmatrix}
  g_{A} \begin{pmatrix} \psi_{1} \\ \psi_{2} \end{pmatrix}$.

\begin{rema}
Observe that, not incidentally, by recalling the proof of \ref{prop:isotropy} we could have computed $ \langle \zeta, \zeta \rangle_{A} $ in a more roundabout way by using the
equation  contained there \[ \langle \zeta, \zeta\rangle_A = D^2(\log\Lambda)_A(\zeta,\zeta) = D(G_\zeta)_A(\zeta)  .\]
As a side effect we easily obtain that
\begin{equation}\label{lambda''0}
D^2(\log\Lambda)_A(\zeta,\zeta)  = \frac{x(1-y)}{(1-x)(2-x-y)} \zeta_{11}^2 + \frac{(1-x)y}{(1-y)(2-x-y)} \zeta_{22}^2
\end{equation}
We see for example that when $x$ or $y$ goes to $0$, the pressure
becomes very flat (as opposed to very convex, i.e. its Hessian goes to
zero).
\end{rema}
From the metric tensor, we compute the curvature at each point. For simplicity, if we let
$ g_A = \left( \begin{array}{cc}
                              E & 0 \\
                              0 & G \\
                            \end{array} \right)$
then we use the explicit formula for the curvature
$$ K(A) =  -  \frac{1}{ 2 \sqrt{EG}} \left\{ \left(\frac{E_{y }}{ \sqrt{EG} }\right)_{y} + \left(\frac{G_{x }}{ \sqrt{EG} }\right)_{x}  \right\}. $$
where subscripts indicate partial derivatives with respect to the indicated variables. The expression simplifies greatly (see Section \ref{app:metric}):
\begin{coro}
The Gaussian curvature of $g$ at $A$ is given when $L_A=S(x,y)$ by
\begin{equation*}\label{scalarcurvature}
K(A) = \frac{1}{(2 - x - y)}
\end{equation*}
\end{coro}

\begin{rema}
In the case at hand, the curvature its always strictly positive.
In fact, it is even bounded away from $0$, so that $\fspace{N}_2$
endowed with $g$ is not complete (indeed, if $g$ where complete
then the Bonnet-Myers theorem would imply that $\fspace{N}_2$ is compact).
\end{rema}

\subsection{Rescaling the metric} \label{re}

We considered in the previous reasoning the Riemmanian norm $\langle \zeta, \zeta \rangle_{A}$ of a tangent vector $\zeta$ at the potential $A$ given by the asymptotic variance, as in theorem \ref{theo:metric}. We wonder how rescaling the metric by the entropy would effect such curvature, based on previous work by McMullen \cite{Mac}.  Given the eigenvector $\pi$ of the previous section (which corresponds to the eigenmeasure) the entropy
is given as a function of $x, y$ by
\[ \begin{split}
 h(x, y)  = & -  \frac{1-y}{2-x-y}\left( x\log(x) + (1 - x)\log(1 - x)  \right)   \\
 & -  \frac{1-x}{2-x-y} \left( (1 - y) \log(1 - y) +  y \log(y) \right)
\end{split} \]

This function is always positive on $(0,1)\times(0,1)$ and is $0$ in the limit to the vertex $(0,0)$ and the edges $\{1\}\times [0,1]$
and $[0,1]\times\{1\}$ (Figure \ref{fig:entropy}). Note that there is
a strong asymmetry between the cases $x=0$ and $x=1$ (similarly for $y$),
as $x=1$ means the Markov chain gets stuck at the state $1$, while $x=0$
means the random walk is always repelled away from state $1$, but then can either stay at $2$ or come back to $1$, leaving enough uncertainty to yield
positive entropy.

\begin{figure}[htp]
\begin{center}
\includegraphics[width=.95\textwidth]{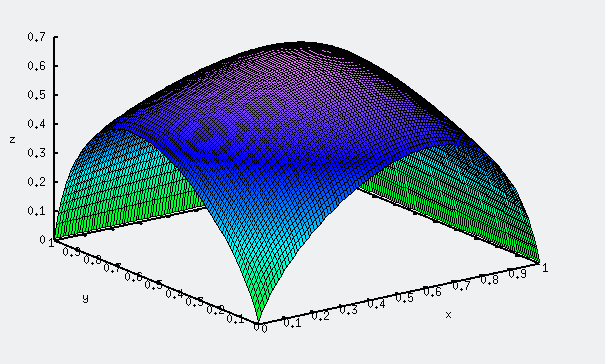}
\caption{The entropy in $(x,y)$ coordinates.}\label{fig:entropy}
\end{center}
\end{figure}
We rescale the metric associated to the matrix $g_A$ of the previous section to a new  $\tilde g_A$ in the interior of the square by setting
$\tilde g_A = \begin{pmatrix}\frac{E}{h} & 0 \\  0 &
\frac{G}{h}\end{pmatrix}$
where $h$ is the entropy functional.
We denote $K$ the curvature  associated to the metric $g$ and
$\widetilde{K}$ the one associated to $\tilde g$.

After a little bit of juggling with the equations, for the strictly positive function $h(x, y)$  one gets
\[ \begin{split}
 \frac{\widetilde{K}}{h} & = K + \frac{1}{ 2 \sqrt{EG}} \left(  \left(\frac{\sqrt{E}}{\sqrt{G}} \frac{h_y}{h}  \right)_{y} +
 \left(\frac{\sqrt{G}}{\sqrt{E}} \frac{h_{x} }{h}  \right)_{x} \right) \\
\end{split} \]

The explicit expression of $\widetilde{K}$ is particularly long and tedious to handle. We use the software Maxima both to do the necessary symbolic manipulation   and the plot of the graph (Figure \ref{fig:curvature-MM}).
In the case of some subshifts related to
Fuchsian groups, McMullen showed that this
precise scaling of the metric identifies with the Weyl-Peterson metric on Teichm\"uller space, which is known to be of negative Ricci curvature.
One could thus expect that $\tilde{g}$ has negative curvature, but
this turns out not to be the case: $\widetilde{K}$
takes both positive and negative values.

\begin{figure}[htp]
\begin{center}
\includegraphics[width=.95\textwidth]{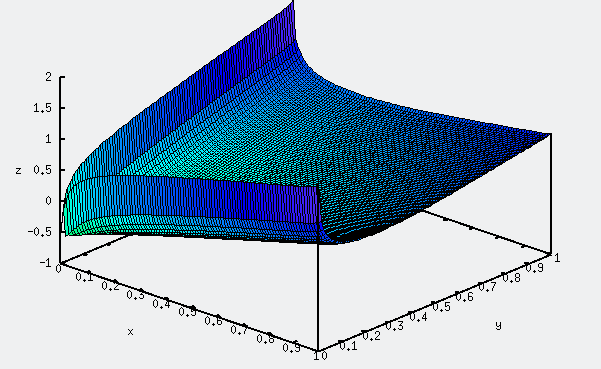}
\caption{Curvature of the variance metric with McMullen's normalization.}
\label{fig:curvature-MM}
\end{center}\end{figure}

\subsection{Intermediate steps} \label{app:metric}

From Section 4 of \cite{Carmo}, to explicitly compute the curvature  we have the followings step.
Observe that
\[
E_{y }  = - \frac{ 1 }{x (2 - x - y)^2 } \quad \mbox{ and }
\quad G_{x } =   - \frac{ 1}{ y(2 - x - y)^2}.
\] Moreover,
$$\sqrt{E\,G} = \frac{1}{ \sqrt{x \, y}(2 - x - y)}.$$ Therefore,
\[
\begin{split}
\frac{\fr}{\fr y} \bigg(\frac{E_{y }}{ \sqrt{EG} }\bigg) & = - \frac{\fr}{\fr y}   \bigg(\frac{ \sqrt{y} }{ \sqrt{x }}\frac{1}{2 - x - y   } \bigg)  = -\frac{1}{2\sqrt{xy}} \frac{(2 - x + y)}{(2 - x - y)^2}
\end{split}
\] and similarly
\[
\begin{split}
\frac{\fr}{\fr x} \bigg(\frac{G_{x }}{ \sqrt{EG} }\bigg) & = -\frac{1}{2\sqrt{xy}} \frac{(2 + x - y)}{(2  - x - y)^2}
\end{split}
\] Finally,
\[
\begin{split}
 \bigg( \frac{E_{y }}{ \sqrt{EG} }\bigg)_{y} + \bigg( \frac{G_{x }}{ \sqrt{EG} }\bigg)_{x} & =
 -\frac{1}{2\sqrt{xy}} \frac{(2 - x + y)}{(2 - x - y)^2} -\frac{1}{2\sqrt{xy}} \frac{(2 + x - y)}{(2  - x - y)^2} \\
& = -\frac{2}{\sqrt{xy}} \frac{1}{(2 - x - y)^2}
\end{split}
\]

\bibliographystyle{smfalpha}
\bibliography{normalization}

\end{document}